\pdfoutput=1\relax
\documentclass[reqno]{amsart}

\usepackage{verbatim}
\usepackage[textsize=scriptsize]{todonotes}
\usepackage{tikz-cd}
\usepackage{etoolbox}
\usepackage{etex}
\usepackage{hyperref}
\usepackage{cleveref}
\usepackage[T1]{fontenc}

\usepackage{chemarr}
\usepackage{amssymb}
\usepackage{amsmath}
\usepackage{comment}
\usepackage{spectralsequences}
\usepackage{cleveref}
\usepackage{mathtools}
\usepackage{rotating}
\usepackage{wrapfig}
\usepackage{outlines}
\usepackage{graphicx}

\theoremstyle{definition}
\newtheorem{nul}{}[section]
\newtheorem{dfn}[nul]{Definition}

\newtheorem{rmk}[nul]{Remark}

\newtheorem{cnstr}[nul]{Construction}
\newtheorem{cnv}[nul]{Convention}

\newtheorem{exm}[nul]{Example}

\newtheorem{rec}[nul]{Recollection}

\newtheorem{qst}[nul]{Question}
\newtheorem{prb}[nul]{Problem}
\newtheorem*{dfn*}{Definition}
\newtheorem*{axm*}{Axiom}
\newtheorem*{ntn*}{Notation}
\newtheorem*{exm*}{Example}
\newtheorem*{exr*}{Exercise}
\newtheorem*{int*}{Intuition}
\newtheorem*{qst*}{Question}
\newtheorem*{rmk*}{Remark}

\theoremstyle{plain}

\newtheorem{thm}[nul]{Theorem}
\newtheorem{prop}[nul]{Proposition}
\newtheorem{lem}[nul]{Lemma}

\newtheorem{cnj}[nul]{Conjecture}
\newtheorem{cor}[nul]{Corollary}

\newtheorem*{thm*}{Theorem}
\newtheorem*{prop*}{Proposition}
\newtheorem*{cor*}{Corollary}
\newtheorem*{lem*}{Lemma}
\newtheorem*{cnj*}{Conjecture}

\let\oldwidetilde\widetilde
\protected\def\widetilde{\oldwidetilde}

\DeclareMathOperator{\coker}{\mathrm{coker}}

\DeclareMathOperator{\Ss}{\mathbb{S}}

\DeclareMathOperator{\F}{\mathbb{F}}

\DeclareMathOperator{\MO}{\mathrm{MO}}
\DeclareMathOperator{\MOfn}{\mathrm{MO \langle 4m \rangle}}
\DeclareMathOperator{\MOn}{\mathrm{MO \langle n \rangle}}

\DeclareMathOperator{\Ext}{\mathrm{Ext}}

\DeclareMathOperator{\GL}{\mathrm{GL}}

\DeclareMathOperator{\Sq}{\mathrm{Sq}}

\newcommand{\sOf}{\Sigma^{\infty} \mathrm{O \langle 4m-1\rangle}}

\newcommand{\wt}{\widetilde}

\newcommand{\BO}{\mathrm{BO}}

\newcommand{\Oo}{\mathrm{O}}

\newcommand{\HFp}{\mathbb{F}_p}
\newcommand{\HFt}{\mathbb{F}_2}

\newcommand{\kappabar}{\overline{\kappa}}

\newcommand{\BSO}{\mathrm{BSO}}
\newcommand{\Str}{\mathrm{String}}
\newcommand{\MStr}{\mathrm{MString}}
\newcommand{\MOtw}{\mathrm{MO}\langle 12 \rangle}
\newcommand{\BStr}{\mathrm{BString}}
\newcommand{\Spin}{\mathrm{Spin}}
\newcommand{\MSpin}{\mathrm{MSpin}}
\newcommand{\ord}{\mathrm{ord}}
\newcommand{\MF}{\mathrm{MF}}
\newcommand{\tmf}{\mathrm{tmf}}
\newcommand{\Ah}{\widehat{A}}
\newcommand{\CC}{\mathbb{C}}
\newcommand{\num}{\mathrm{num}}
\newcommand{\denom}{\mathrm{denom}}
\newcommand{\sig}{\mathrm{sig}}
\newcommand{\bp}{\mathrm{bP}}
\newcommand{\Z}{\mathbb{Z}}
\newcommand{\ZZ}{\mathbb{Z}}

\DeclarePairedDelimiter\abs{\lvert}{\rvert}%
\makeatletter
\let\oldabs\abs
\def\abs{\@ifstar{\oldabs}{\oldabs*}}

\usepackage{tikz}
\usetikzlibrary{matrix,arrows,decorations}
\usepackage{tikz-cd}

\usepackage{adjustbox}

\let\oldtocsection=\tocsection
 
\let\oldtocsubsection=\tocsubsection
 
\let\oldtocsubsubsection=\tocsubsubsection
 
\renewcommand{\tocsection}[2]{\hspace{0em}\oldtocsection{#1}{#2}}
\renewcommand{\tocsubsection}[2]{\hspace{1em}\oldtocsubsection{#1}{#2}}
\renewcommand{\tocsubsubsection}[2]{\hspace{2em}\oldtocsubsubsection{#1}{#2}}

\usepackage{etoolbox}
\newtoggle{draft}
\togglefalse{draft}
\iftoggle{draft} {
\usepackage[margin=1.5in]{geometry}
\newcommand{\NB}[1]{\todo[color=gray!40]{#1}}
\newcommand{\TODO}[1]{\todo[color=red]{#1}}
}{ % else
\usepackage[margin=1.5in]{geometry}
\newcommand{\NB}[1]{}
\newcommand{\TODO}[1]{}
\renewcommand{\todo}[1]{}
\renewcommand{\todo}[1]{}
}
\title{On the high-dimensional geography problem}
\author{Robert Burklund}
\address{Department of Mathematics, MIT, Cambridge, MA, USA}
\email{burklund@mit.edu}

\author{Andrew Senger}
\address{Department of Mathematics, MIT, Cambridge, MA, USA}
\email{senger@mit.edu}

\begin{document}
\begin{abstract}
  %% Let $n \geq 3$. Wall showed that smooth, closed, oriented $(n-1)$-connected $2n$-manifolds are classified up to connected sum with an exotic $2n$-sphere by certain algebraic data he calls an $n$-space.
  %% In this paper, we give a complete answer to which $n$-spaces are realized by $(n-1)$-connected $2n$-manifolds for all $n \neq 63$, where the answer depends on the resolution of the Kervaire invariant one problem.
  %% Along the way, we completely resolve conjectures of Galatius--Randal-Williams and Bowden--Crowley--Stipsicz, showing that they are true outside of the exceptional dimension $24$, where we provide a counterexample. This counterexample is related to the Witten genus and its refinement to a map of $E_\infty$-ring spectra by Ando--Hopkins--Rezk.

  %% By previous work of many authors, including Wall, Schultz, Stolz and Hill--Hopkins--Ravenel, as well as recent joint work of Hahn with the authors, these questions have been resolved for all but finitely many dimensions, and the contribution of this paper is to fill in these gaps.

  In 1962, Wall showed that smooth, closed, oriented, $(n-1)$-connected $2n$-manifolds of dimension at least $6$ are classified up to connected sum with an exotic sphere by an algebraic refinement of the intersection form which he called an $n$-space.

  In this paper, we complete the determination of which $n$-spaces are realizable by smooth, closed, oriented, $(n-1)$-connected $2n$-manifolds for all $n \neq 63$. In dimension $126$  the Kervaire invariant one problem remains open.
  Along the way, we completely resolve conjectures of Galatius--Randal-Williams and Bowden--Crowley--Stipsicz, showing that they are true outside of the exceptional dimension $23$, where we provide a counterexample. This counterexample is related to the Witten genus and its refinement to a map of $\mathbb{E}_\infty$-ring spectra by Ando--Hopkins--Rezk. 

  By previous work of many authors, including Wall, Schultz, Stolz and Hill--Hopkins--Ravenel, as well as recent joint work of Hahn with the authors, these questions have been resolved for all but finitely many dimensions, and the contribution of this paper is to fill in these gaps.

    %In the paper "On the boundaries of highly connected, almost closed manifolds," the $n \geq 32$ case of the following conjecture of Galatius and Randal-Williams was proven: any $(8n-1)$-dimensional homotopy sphere which bounds a $(4n-1)$-connected manifold also bounds a parallelizable manifold. Here, we show that the conjecture also holds for $4 \leq n \leq 16$. TODO: Check that we can get this high at $3$. The cases $17 \leq 31$ would follow from further knowledge of the $3$-primary Adams spectral sequence.
%
%    Interestingly, we find that the conjecture fails for $n = 3$, and we completely analyze the extent to which it does so.
%
%    Mention applications...
\end{abstract}
\maketitle

\setcounter{tocdepth}{1}
\tableofcontents
\vbadness 5000

%----------------------------------------------------------------------%

% \newpage
\section{Introduction} \label{sec:intro}
A classical problem in differential topology is the following:

\begin{prb} \label{prb:main}
  Classify, or enumerate, all smooth, closed, oriented, $(n-1)$-connected manifolds of dimension $2n$.
\end{prb}

Early progress includes both Adams's solution to the Hopf invariant one problem \cite{AdamsHopf} and Milnor's discovery of exotic spheres \cite{MilnorTalk}. A major advance was made by Wall in his 1962 paper \cite{Wall62}. Wall showed that the diffeomorphism type of a smooth, closed, oriented, $(n-1)$-connected $2n$-manifold of dimension at least $6$ is determined, up to connected sum with a homotopy sphere, by the middle homology group, the intersection pairing, and the so-called normal bundle data.\footnote{The restriction to dimension at least $6$ is inherited from the the use of the Whitney trick in Smale's study of handlebody decompositions.} Wall refers to such a collection of algebraic invariants as an \emph{$n$-space}. The precise definition of an $n$-space will be given in \Cref{sec:class}. To enumerate all smooth, closed, oriented, $(n-1)$-connected $2n$-manifolds in terms of $n$-spaces, it therefore suffices to answer the following two questions:

\begin{enumerate}
  \item Which $n$-spaces may be realized by smooth, closed, oriented, $(n-1)$-connected $2n$-manifolds?
  \item Given an $n$-space which is realized by a manifold $M$, for which homotopy spheres $\Sigma$ are $\Sigma \# M$ and $M$ diffeomorphic?
\end{enumerate}

In analogy with the study of smooth structures on simply-connected four dimensional manifolds, we refer to these as the high-dimensional \textbf{geography} and \textbf{botany} problems.
Following a great deal of work over the past half-century, both the high-dimensional geography and botany problems have been resolved in all but finitely many dimensions \cite{Wall62,KervaireMilnor,BPKervaire,Wall67,Kosinski,MahTanDiff,BrowderKervaire,Schultz,Lampe,BJMKervaire,StolzBook,StolzbP,HHR,Boundaries}.

Developments after 1987 include the work of Hill--Hopkins--Ravenel on the Kervaire invariant one problem, which, when combined with work of Stolz, settled the geography problem for all odd $n > 135$, and work of Hahn and the authors settling the geography problem when $n > 124$ is a multiple of $4$, as well as the botany problem when $n > 232$ is congruent to $1$ modulo $8$.

In this paper, we complete the solution to the high-dimensional geography problem outside of dimension $126$, where the answer is contingent on the resolution of the Kervaire invariant one problem.
Our answer is phrased in terms of certain $n$-space invariants studied by Wall.
Although some of these invariants such as the signature and the Kervaire invariant, $\Phi$, are likely to be familiar to the reader, others such as the middle-homology-class $\chi$ might not be. In \Cref{sec:class} we recall the definitions for all of the invariants we use.

%% \begin{thm} [Proved as \Cref{thm:geography}] \label{thm:geographyintro}
%%   Suppose that $3 \leq n \neq 63$ is an integer. An $n$-space is realized by a closed $(n-1)$-connected $2n$-manifold if and only if the following conditions hold:
%%   \begin{itemize}
%%     \item If $n = 4,8$, then 
%%       \[\frac{\sig - \chi^2}{8} \equiv 0 \mod \frac{\sigma_{n/2}}{8}.\]
%%     \item If $n \equiv 0 \mod 4$ and $n \geq 12$, then
%%       \[\frac{\sig}{8} + \frac{\chi^2}{2} s(Q)_{n/2} \equiv 0 \mod \frac{\sigma_{n/2}}{8}.\]
%%       If $n = 12$, then we demand further that
%%       \[\frac{\chi^2}{2} \equiv 0 \mod 2.\]
%%     \item If $n \equiv 2 \mod 4$, then
%%       \[\frac{\sig}{8} \equiv 0 \mod \frac{\sigma_{n/2}}{8}.\]
%%     \item If $n \equiv 1 \mod 2$ and $n \neq 3,7,15,31,63$, then
%%       \[\Phi = 0,\]
%%       where $\Phi$ is defined as in \Cref{dfn:ker-inv} when $n \equiv 1 \mod 8$.\footnote{When $n \equiv 1 \mod 8$, Wall only defines $\Phi$ \cite{Wall62} up to the indeterminacy $\varphi(\chi)$. In \Cref{dfn:ker-inv}, we specify which choice of $\Phi$ must be made for \Cref{thm:geographyintro} to hold.}
%%       \newline If $n = 9$, then we demand further that
%%       \[\varphi(\chi) = 0.\]
%%     \item If $n = 3,7,15,31$, then every $n$-space is realizable.
%%   \end{itemize}
%% \end{thm}
\begin{thm} [Proven as \Cref{thm:geography}] \label{thm:geographyintro}
  Suppose that $ n \geq 3 $.
  With the exception of finitely many $n$, an $n$-space $(H, H \otimes H \to \Z, \alpha)$ is realized by a smooth, closed, oriented, $(n-1)$-connected $2n$-manifold if and only if the following conditions hold:
  \begin{enumerate}
    \item If $n \equiv 0 \mod 4$, then
      $ \sig + 4 s(Q)_{n/2} \chi^2 \equiv 0 \mod \sigma_{n/2} $.
    \item If $n \equiv 2 \mod 4$, then
      $\sig \equiv 0 \mod \sigma_{n/2} $.
    \item If $n \equiv 1 \mod 2$, then $\Phi = 0$.
  \end{enumerate}
  The full list of exceptions is as follows:
  \begin{itemize}
  \item If $n = 3,7,15,31$, then every $n$-space is realizable.
  \item If $n = 63$, then every $n$-space is realizable if there exists a closed smooth manifold of Kervaire invariant one in dimension $126$. Otherwise, an $n$-space is realizable if and only if $\Phi = 0$.
  \item If $n = 4$ or $8$, then instead of condition (1) we require $\sig - \chi^2 \equiv 0 \mod \sigma_{n/2}$.
  \item If $n = 9$, we require condition (3) and demand further that $\varphi(\chi) = 0$.
  \item If $n = 12$, we require condition (1) and demand further that $\chi^2 \equiv 0 \mod 4$.
  \end{itemize}
\end{thm}

In his 1962 paper, Wall showed for $n \geq 3$ that $n$-spaces lie in bijection with diffeomorphism classes of oriented, $(n-1)$-connected smooth $2n$-manifolds with boundary a homotopy sphere. An $n$-space may be realized precisely when this homotopy sphere may be filled in, i.e. when it is diffeomorphic to the standard $(2n-1)$-sphere. Therefore, the high-dimensional geography problem is intimately related to the following question:
%is more-or-less equivalent to the following question:

%% In the study of the geography problem for highly connected manifolds, it is therefore important to consider the following question:

\begin{qst}\label{qst:main}
  Given an integer $n > 2$, which $(2n-1)$-dimensional homotopy spheres arise as the boundary of an $(n-1)$-connected $2n$-manifold?
\end{qst}

The answer to \Cref{qst:main} invokes knowledge of the Kervaire--Milnor group of homotopy spheres, so we begin by recalling its basic structure \cite{KervaireMilnor}.
%% To state the answer to \Cref{qst:main}, we must recall some knowledge of the structure of the group of homotopy spheres.
%% Before we can provide the answer to \Cref{qst:main}, we must recall the structure of the group of homotopy spheres.
%% For $m > 4$, let $\Theta_m$ denote the Kervaire--Milnor group of closed, smooth, oriented manifolds $\Sigma$ that are homotopy equivalent to the $n$-sphere, where the group operation is the connected sum \cite{KervaireMilnor}. The Kervaire--Milnor exact sequence
For $m > 4$, let $\Theta_m$ denote the group of oriented, smooth, closed manifolds $\Sigma$ that are homotopy equivalent to the $m$-sphere, where the group operation is the connected sum. The Kervaire--Milnor exact sequence
\[0 \to \bp_{m+1} \to \Theta_{m} \to \coker(J)_{m}\]
expresses $\Theta_m$ in terms of the finite cyclic group $\bp_{m+1}$ and the much more complicated group $\coker(J)_{m}$. Geometrically, the group $\bp_{m+1}$ is the subgroup of $\Theta_{m}$ consisting of those homotopy spheres which bound parallelizable manifolds.

\begin{thm} \label{thm:main}
  Suppose that $n > 2$ and $n \neq 9, 12$.
  Then a $(2n-1)$-dimensional homotopy sphere is the boundary of an $(n-1)$-connected smooth $2n$-manifold if and only if it also bounds a parallelizble manifold.

  A homotopy $17$-sphere $\Sigma$ is the boundary of an $8$-connected $18$-manifold if and only if
  \[ [\Sigma] \in \{0, \eta \eta_4\} \subset \coker(J)_{17}. \]
  A homotopy $23$-sphere $\Sigma$ is the boundary of an $11$-connected $24$-manifold if and only if
  \[ [\Sigma] \in \{0, \eta^3 \kappabar\} \subset \coker(J)_{23}. \]
\end{thm}

\begin{rmk} \label{rmk:new}
  \Cref{thm:main} is new in the the following cases: 
  \begin{itemize}
    \item When $n \equiv 0 \mod 4$ and $12 \leq n \leq 124$.
    \item When $n \equiv 1 \mod 8$ and $9 \leq n \leq 129$.
  \end{itemize}
  In the remaining cases the attribution is as follows:
  \begin{itemize}
  \item When $3 \leq n \leq 8$, it is due to Wall \cite{Wall62}.
  \item When $n \equiv 3,5,6,7 \mod 8$, it is a corollary of work of Wall and Kervaire--Milnor \cite{Wall62} \cite{KervaireMilnor}.
  %% \item When $n \equiv 5 \mod 8$, it is due to Brown and Peterson \cite{BPKervaire}.
  %% \item When $n \equiv 3 \mod 8$, it is due to Browder \cite{BrowderKervaire}.
  \item When $n \equiv 2 \mod 8$, it is due to Schultz \cite[Corollary 3.2 and Theorem 3.4(i)]{Schultz}.
  %% \item When $n \equiv 7 \mod 8$ and $n \neq 2^k-1$ it is due to ???.
  %% \item When $n = 31$ it is due to Barratt, Jones, and Mahowald \cite{BJMKervaire} (cf. \cite{XuKervaire}).
  \item When $n \equiv 1 \mod 8$ and $n \geq 129$, it is due to Stolz \cite[Theorem B]{StolzBook}.\footnote{In \cite[Theorem B]{StolzBook}, Stolz claims this result for all $n \geq 113$, and in this case Stolz's proof is in fact valid for $n \geq 105$. However, Stolz made crucial use of a theorem announced by Mahowald, whose statement appears in his work as \cite[Satz 12.9]{StolzBook}, of which no proof has appeared in the literature. A similar theorem was proven by Hahn and the authors \cite[Section 15]{Boundaries}, which when plugged into Stolz's argument gives the result for $n \geq 129$.}% \todo{Robert: Check my numbers on this footnote! Satz 12.3 of Stolz is the main one, also see Satz 12.9. } \todo{I checked the numbers and got a different answer. I think the strict v.s. weak inequalities work out so that I can replace the number 18 in Stolz 12.9 with 25 from 15.10 + 15.8 in our paper. I've pushed a program that does the computation and the case 129 seems fine. Given that we disagree you should probably check again.}
  %% \item When $n = 2^k - 1$ and $k \geq 6$, it is due to Hill, Hopkins and Ravenel \cite{HHR}.
  \item When $n \equiv 0 \mod 4$ and $n \geq 128$ it is due to Hahn and the authors \cite[Theorem 8.5]{Boundaries}.
  \end{itemize}
  
  %\Cref{thm:main} was already known in the cases $n=4,8$ and (ii) $n \equiv 1 \mod 8$, $n \geq 113$. \todo{Modulo Stolz citing Mahowald...} For $n=4,8$, it is due to Wall \cite{Wall62}, and for $n \equiv 1 \mod 8, n \geq 118$, it is due to Stolz \cite[Theorem B]{StolzBook}.
\end{rmk}
%
%\begin{rmk}
%  In \Cref{thm:main}, we did not resolve what happens in the case $n=9$.
%  Preliminary calculations suggest the following result in this case: a $17$-dimensional homotopy sphere $\Sigma \in \Theta_{17}$ also bounds a parallelizable manifold if and only if $[\Sigma] = 0$ or $\nu \kappa \in \coker(J)_{17}$.
%\end{rmk}
%
%\begin{rmk}
%  When combined with the main result of \cite{Boundaries}, \Cref{thm:main} leaves open some middle dimensional cases.
%  Here we are limited by our knowledge of the $3$-primary Adams spectral sequence.
%  And while it is in principle easy to make the calculations that we need in this spectral sequence, we found that such a computation would be too tedious to write up.
%  In future work, the second named author will prove a substantially improved vanishing line in the $3$-primary ASS, and the result will follow from this.
%\end{rmk}
%

%\Cref{thm:main} resolves Conjectures A and B of \cite{GRAbelianQuotients}, which conjecture that if $n \equiv 0 \mod 4$ then a $(2n-1)$-dimensional homotopy sphere
\Cref{thm:main} resolves the following conjecture of Galatius and Randal-Williams, which is equivalent to Conjectures A and B of \cite{GRAbelianQuotients}:

\begin{cnj} \label{cnj:G-RW}
  Suppose that $n \equiv 0 \mod 4$. Then a $(2n-1)$-dimensional homotopy sphere is the boundary of an $(n-1)$-connected $2n$-manifold if and only if it also bounds a parallelizable manifold.
\end{cnj}

In particular, we learn that the conjecture is is false for $n = 12$ and true otherwise.
The interest of Galatius and Randal-Williams in this conjecture was spurred on by their work on mapping class groups of highly connected manifolds.
In \Cref{sec:mapping}, we will briefly record some applications of \Cref{thm:main} to the computation of mapping class groups.

As we shall see in \Cref{sec:stein}, \Cref{thm:main} also resolves the following conjecture of Bowden, Crowley and Stipsicz:

\begin{cnj}[{\cite[Conjecture 5.9]{SteinFillable}}]
  An odd dimensional homotopy sphere admits a Stein fillable contact structure if and only if it also bounds a parallelizable manifold.
\end{cnj}

More precisely, we show in \Cref{thm:stein} that the conjecture is true for all dimensions other than $23$, and that a $23$-dimensional homotopy sphere $\Sigma$ admits a Stein fillable contact structure if and only if
\[ [\Sigma] \in \{0, \eta^3 \kappabar\} \subset \coker(J)_{23}. \]

In \cite[Proposition 5.3]{SteinFillable}, Bowden, Crowley and Stipsicz show how to equip an odd-dimensional homotopy sphere which bounds a parallelizable manifold with an explicit, geometrically defined Stein fillable contact structure. This leads us to ask the following question:
%It would be very interesting, and likely shed light on what is special about dimension $23$, to find a geometric source for the Stein fillable contact structure on homotopy $23$-spheres $\Sigma$ with $[\Sigma] = \eta^3 \kappabar \in \coker(J)_{23}$.
%
\begin{qst}
  Given an exotic $23$-sphere $\Sigma$ with $[\Sigma] = \eta^3 \kappabar \in \coker(J)_{23}$, can one construct an explicit Stein fillable contact structure on $\Sigma$ in a geometric way? Can this be done in such a way to shed light on what is special about dimension $23$ and the class $\eta^3 \kappabar \in \coker(J)_{23}$?
\end{qst}

Let us take a moment to discuss why the dimension $23$ is exceptional, providing counterexamples to both the conjectures of Galatius--Randal-Williams and Bowden--Crowley--Stipsicz. We proceed by contradiction: supposing that \Cref{cnj:G-RW} held for $n=12$, the work of Wall \cite{Wall62} implies the existence of a closed, oriented, $11$-connected smooth $24$-manifold with certain Pontryagin numbers. This manifold would have to admit a string structure, and so we may consider its Witten genus, which can be computed in terms of the Pontryagin numbers.
%We will see that the Witten genus of this manifold cannot be the Witten genus of any closed string $24$-manifold.
%The issue which arises is that the value one obtains violates known integrality constriant on the Witten genus which we now discuss. 

The Witten genus of a closed string manifold is an integral modular form. However, not every integral modular form is the Witten genus of a closed string manifold: a nontrivial restriction on the image of the Witten genus is provided by the Ando--Hopkins--Rezk string orientation \cite{AHR}, which implies that the Witten genus factors through the homotopy groups of the connective spectrum $\tmf$ of topological modular forms.\footnote{In fact, a folk theorem of Hopkins and Mahowald, which has now been written up by Devalapurkar \cite{Sanath}, shows that this is the only restriction on the image of the Witten genus: the map $\Omega^{\Str} _{n} = \pi_{n} \MStr \to \pi_{n} \tmf$ is surjective for all $n$.}
For example, the weight $12$ modular form $\Delta$ does not lie in the image of the Witten genus, instead only multiples of $24 \Delta$ lie in the image. Using this restriction, we are able to show that the putative manifold constructed above cannot exist.
% because its Witten genus cannot be the Witten genus of a closed string manifold.%\footnote{It is intesting to note that the analogous integrality property for the $\widehat{A}$-genus of a closed spin manifold, namely that the $\widehat{A}$-genus of $(8n+4)$-dimensional spin manifold is even, is apparent from the interpretation of the $\widehat{A}$-genus as the index of Dirac operator on the spinor bundle, as the $(8n+4)$-dimensional Clifford algebra contains the quaternions. We hope that one day it will be possible to see a similar geometric origin for the restriction on Witten genus of a closed string manifold.}

As far as the authors are aware, thus far there have been few concrete geometric applications of the Witten genus and the Ando--Hopkins--Rezk string orientation.\footnote{However, see \cite{KrannichExotic} for a recent application of the Ando--Hopkins--Rezk orientation to the question of how taking the connected sum with an exotic sphere affects the mapping class group of a highly-connected manifold.}
We were therefore pleased to find an application for this beautiful theory in this work.

\begin{rmk}
  The argument sketched above, whose details are the subject of \Cref{sec:witten}, is modeled on a classical argument making use of the $\Ah$-genus, which shows that there is no closed, simply-connected, smooth $4$-manifold whose intersection form is isomorphic to the $E_8$-form, though Freedman has famously shown the existence of such a topological $4$-manifold \cite{Freedman}.
  Indeed, if such a smooth $4$-manifold existed then it would have to admit a spin structure and its $\Ah$-genus would be equal to $1$.
  But the fact that the $\Ah$-genus factors as the composite 
  \[\Omega^{\Spin}_{4} \cong \pi_{4} \MSpin \to \pi_{4} \mathrm{ko} \to \pi_{4} \mathrm{ku} \cong \Z,\]
  where the first map is induced by the Atiyah-Bott-Shapiro orientation \cite{ABS}, and the second map is induced by tensoring up a real vector bundle to the complex numbers, implies that the $\Ah$-genus of a $4$-manifold is even, since the image of the second map is equal to $2\Z \subset \Z$.
  Since the signature of a simply-connected spin $4$-manifold is equal to $-\frac{1}{8}$ times its $\Ah$-genus, this argument also proves Rokhlin's theorem, which states that the signature of such a manifold must be divisible by $16$. %Since the signature of the $E_8$-form is $8$, this proves the desired result.

  It is interesting to note that the extra integrality for the $\Ah$-genus used above is apparent from the interpretation of the $\widehat{A}$-genus as the index of the Dirac operator on the spinor bundle, as the rank $4$ real Clifford algebra contains the quaternions.
  We hope that one day it will be possible to see a similar geometric origin for the restriction on the Witten genus of a closed string manifold used above. At the moment the appropriate replacement for Clifford algebras is not yet clear, but see \cite{STEll, CNet}.

  %% One consequence of the argument of \Cref{rmk:Witten} is a restriction on which $12$-spaces may be realized by closed $11$-connected smooth $24$-manifolds.
  %% Our argument is entirely analogous to a classical argument, making use of the $\Ah$-genus instead of the Witten genus, which shows that there is no closed, simply-connected, smooth four-manifold whose intersection form is isomorphic to the $E_8$-form.
  %% Indeed, if such a four-manifold existed then it would have to admit a spin structure and its $\Ah$-genus would be equal to $1$.
  %% But the fact that the $\Ah$-genus factors through map on homotopy groups
  %% $\Omega^{\Spin}_{4n} \cong \pi_{4n} \MSpin \to \pi_{4n} ko$
  %% induced by the Atiyah-Bott-Shapiro orientation map \cite{ABS} implies that the $\Ah$-genus of a manifold of dimension congruent to $4$ modulo $8$ must be even.
  %% Since the signature of a simply-connected spin four-manifold is equal to $-\frac{1}{8}$ times its $\Ah$-genus, the same argument proves Rokhlin's theorem, which states that the signature of such a manifold must be divisible by $16$.

  %% It is intesting to note that the extra integrality for the $\Ah$-genus used above is apparent from the interpretation of the $\widehat{A}$-genus as the index of the Dirac operator on the spinor bundle, as the $(8n+4)$-dimensional Clifford algebra contains the quaternions.
  %% We hope that one day it will be possible to see a similar geometric origin for the restriction on the Witten genus of a closed string manifold used in \Cref{rmk:Witten}.
\end{rmk}

\subsection{An outline of the paper}
In \Cref{sec:class}, we deduce \Cref{thm:geographyintro} from \Cref{thm:main}, solving the high-dimensional geography problem for manifolds of dimension other than $126$.
In \Cref{sec:witten}, we exploit the integrality properties of the Witten genus to prove the existence of an exceptional $11$-connected $24$-manifold $M_{24}$ whose boundary $\partial M_{24}$ is a homotopy sphere which does not bound a parallelizable manifold.
%This integrality comes from the Ando--Hopkins--Rezk string orientation, which implies that the Witten genus factors through the homotopy groups of the connective spectrum $\tmf$ of topological modular forms.

In \Cref{sec:homotopy}, we lay the groundwork for the proof of \Cref{thm:main}, and we divide the proof into four parts.
In \Cref{sec:inch}, we improve upon the key argument from \cite[Section 10]{Boundaries} to prove several cases of \Cref{thm:main}. As a consequence, we are also able to show that $[\partial M_{24}] = \eta^3 \kappabar \in (\coker(J)_{23})_{(2)}$.
In \Cref{sec:2-9}, we use power operations to prove the existence of the exceptional almost closed manifolds in dimension $18$.
In \Cref{sec:3-3}, we again exploit the Ando-Hopkins-Rezk string orientation, this time in order to resolve the $3$-primary aspects of the $24$-dimensional case.
In \Cref{sec:medium}, we make a short homological algebra argument necessary to finish the $128$-dimensional case of \Cref{thm:main}.

Finally, in \Cref{sec:app}, we briefly discuss the applications of our work to Stein fillable homotopy spheres and mapping class groups of highly connected manifolds.

\subsection{Acknowledgments}
The authors would like to express their thanks to Manuel Krannich for introducing them to the conjecture of Galatius and Randal-Williams. %\cite[Conjectures A \& B]{GRAbelianQuotients}
The authors would also like to thank Jeremy Hahn and Zhouli Xu for helpful conversations regarding the contents of this paper. They would further like to thank Diarmuid Crowley, Jeremy Hahn, Manuel Krannich and Haynes Miller for useful comments on a draft of this paper.
\todo{Others?}

During the course of this work, the second named author was supported by an NSF GRFP fellowship under Grant No. 1122374.

\section{Classification of $(n-1)$-connected $2n$-manifolds} \label{sec:class}
In this section, we will reduce the high-dimensional geography problem to \Cref{thm:main} and state the answer as \Cref{thm:geography}. While many cases of \Cref{thm:geography} were previously known, we hope that the reader will find it useful to have a precise answer collected in a single omnibus theorem. 

\begin{cnv}
  All manifolds in this section and \Cref{sec:witten} will be assumed compact, smooth and oriented. We will let $n$ denote an integer greater than $2$.
  %The phrase highly-connected manifold will mean an $(n-1)$-connected $2n$-manifold and $n$ will always be at least $3$.
\end{cnv}

\subsection{The work of Wall}

Let us begin by recalling Wall's work on the classification of closed, $(n-1)$-connected $2n$-manifolds \cite{Wall62}.
Given such a manifold $M$, Wall associates the following data:

\begin{itemize}
  \item the middle homology group $H = H_n (M; \Z)$, which is a finite-dimensional free abelian group,
  \item the intersection pairing $H \otimes H \to \Z$, which is a unimodular bilinear form, symmetric if $n$ is even and skew-symmetric if $n$ is odd, and
  \item the normal bundle data, which is a map of sets $\alpha : H \to \pi_{n} \BSO(n)$ assigning to $x \in H$ the normal bundle of an embedded sphere representing $x$. We will recall the values of the groups $\pi_n \BSO(n)$ below.
\end{itemize}

%The normal bundle data map $\alpha$ is defined as follows: given an $(n-1)$-connected $2n$-manifold $M$, it follows from the Hurewicz isomorphism and a theorem of Haefliger \cite{Haefliger} that any homology class $x \in H$ map be represented by a smoothly embedding sphere $S^n \hookrightarrow M$, uniquely up to smooth isotopy.
%The oriented normal bundle of such an embedded sphere, which is classified by an element of $\pi_{n-1} \SO(n)$, is thus an invariant of $M$.
%
Let $\tau_{S^n} \in \pi_{n} \BSO(n)$ correspond to the tangent bundle of $S^n$.
Moreover, let $J: \pi_{n} \BSO(n) \to \pi_{2n-1} S^{n}$ denote the unstable $J$-homomorphism and let $H : \pi_{2n-1} S^n \to \Z$ denote the Hopf invariant.
Then the above data satisfy the following compatibility conditions: given any $x,y \in H$, we have
\begin{align} \label{eq:wall1}
  x^2 = HJ\alpha(x) \,\,\,\, \text{ and}
\end{align}
\begin{align} \label{eq:wall2}
  \alpha(x+y) = \alpha(x) + \alpha(y) + (xy)\cdot \tau_{S^n},
\end{align}
where in both cases we have used multiplication to denote the intersection product.

\begin{dfn}
  An $n$-space is a triple $(H, H\otimes H \to \Z, \alpha)$ which satisfies (\ref{eq:wall1}) and (\ref{eq:wall2}). Two $n$-spaces $(H_1, H_1 \otimes H_1 \to \Z, \alpha_1 )$ and $(H_2, H_2 \otimes H_2 \to \Z, \alpha_2 )$ are said to be isomorphic if there is an isomorphism of abelian groups $H_1 \cong H_2$ preserving the intersection forms and normal bundle data.
\end{dfn}

%% \begin{center} \begin{tabular}{|c||c|c|c|c|c|c|c|c|} \hline
%%   $n$ & $8s$ & $8s + 1$ & $8s + 2$ & $8s + 3$ & $8s + 4$ & $8s + 5$ & $8s + 6$ & $8s + 7$ \\\hline
%%   $\pi_{n}\mathrm{BSO}(n)$ & $\Z \oplus \Z$ & $\Z/2 \oplus \Z/2$ & $\Z \oplus \Z/2$ & $\Z/2$ & $\Z \oplus \Z$ & $\Z/2$ & $\Z$ & $\Z/2$ \\\hline
%% \end{tabular} \end{center}

Wall proved that the $n$-space of a closed, $(n-1)$-connected $2n$-manifold $M$ determines $M$ up to connected sum with a homotopy sphere:

\begin{thm}[{\cite{Wall62}}] \label{thm:in}
  If two closed, $(n-1)$-connected $2n$-manifolds $M$ and $N$ have isomorphic $n$-spaces, then $M \cong N \# \Sigma$ for some homotopy sphere $\Sigma \in \Theta_{2n}$.
\end{thm}

Moreover, call a manifold almost closed if its boundary is a homotopy sphere.
Then, one may equally well associate an $n$-space to an almost closed, $(n-1)$-connected $2n$-manifold.
Wall's invariant is even more powerful in this case.
%shows that almost closed $(n-1)$-connectd $2n$-manifolds are in one-to-one correspondence with $n$-spaces:

\begin{thm}[\cite{Wall62}]
  The map associating an $n$-space to an almost closed, $(n-1)$-connected $2n$-manifold induces a bijection between the set of almost closed, $(n-1)$-connected $2n$-manifolds up to diffeomorphism and the set of $n$-spaces up to isomorphism.
\end{thm}

Together, these two theorems reduce the classification of $(n-1)$-connected $2n$-manifolds to the following two questions, which we have named in analogy with the classification of simply-connected smooth $4$-manifolds:
\begin{enumerate}
\item \textbf{(High-dimensional Geography problem)}
  Given an $n$-space $(H, H \otimes H \to \Z, \alpha)$, when is it realized by a closed, $(n-1)$-connected $2n$-manifold $M$? This is equivalent to asking when the boundary of the associated almost closed manifold is diffeomorphic to the standard $(2n-1)$-sphere.
\item \textbf{(High-dimensional Botany problem)}
  Given an $n$-space $(H, H \otimes H \to \Z, \alpha)$ which is realized by a $(n-1)$-connected $2n$-manifold $M$, what is the subgroup $I(M) \subset \Theta_{2n}$ of $\Sigma$ such that $M \# \Sigma \cong M$?
\end{enumerate}

Later, Wall gave a cobordism interpretation of the remaining aspects of the high-dimensional geography problem. Since the boundary is preserved by cobordisms of almost closed, $(n-1)$-connected $2n$-manifolds which restrict to $h$-cobordisms on the boundary, the problem may be summed up in an exact sequence of cobordism groups.

\begin{dfn}
  Let $\Omega^{\langle n-1 \rangle} _{2n}$ denote the group of closed, oriented, $(n-1)$-connected $2n$-manifolds, modulo $(n-1)$-connected oriented cobordisms.

  Furthermore, let $A^{\langle n-1\rangle} _{2n}$ denote the group of oriented, almost closed, $(n-1)$-connected $2n$-manifolds, modulo $(n-1)$-connected, oriented cobordisms restricting to $h$-cobordisms on the boundary.
\end{dfn}

\begin{prop}[{\cite[Lemma 32]{Wall67}}]
    There is an exact sequence
    \[\Theta_{2n} \to \Omega^{\langle n-1 \rangle} _{2n} \to A^{\langle n-1 \rangle} _{2n} \xrightarrow{\partial} \Theta_{2n-1},\]
    where the first map sends a homotopy sphere to its cobordism class, the second map cuts out the interior of a smoothly embedded $2n$-disk, and the last map sends an almost closed manifold to its boundary.
\end{prop}

 The high-dimensional geography problem is thus equivalent to the computation of the kernel of the map
\[\partial: A^{\langle n-1 \rangle} _{2n} \to \Theta_{2n-1}\]
in terms of the associated $n$-spaces. % \Cref{thm:main} determines the image of $\partial$.
The results of this paper, building upon a great deal of work in the literature, compute the map $\partial$ and thereby answer the high-dimensional geography problem for all $n \neq 63$. The remaining case $n = 63$ is equivalent to the Kervaire invariant one problem in dimension $126$.

Before proceeding, we find it helpful to unpack the information present in an $n$-space.
First, we note that, since a complete classification of unimodular lattices is not known, the possible bilinear forms are not completely enumerated. This issue will not affect us, but it is worth mentioning.
The definition of an $n$-space depended on the classification of rank $n$ vector bundles on the $n$-sphere and the class of the tangent bundle in that group. We recall Kervaire's work \cite{KervaireUnstable} on this subject. For $n$ at least 8, we have the following table of values:

\begin{center}
  \renewcommand{\arraystretch}{1.2}
  \begin{tabular}{|c||c|c|c|c|c|c|c|c|} \hline
    $n \pmod 8$ & $0$ & $ 1$ & $ 2$ & $ 3$ & $ 4$ & $ 5$ & $ 6$ & $ 7$ \\\hline
    $\pi_{n}\mathrm{BSO}(n)$ & $\Z \oplus \Z$ & $\Z/2 \oplus \Z/2$ & $\Z \oplus \Z/2$ & $\Z/2$ & $\Z \oplus \Z$ & $\Z/2$ & $\Z$ & $\Z/2$ \\\hline
  \end{tabular}
  \label{tbl:BSOn}
  \renewcommand{\arraystretch}{1.0}
\end{center}

Furthermore, the stabilization map to $\pi_n \mathrm{BSO}$ is surjective with kernel generated by $\tau_{S^n}$. In the $n \equiv 1 \mod 8$ case, $\pi_n \mathrm{BSO}(n)$ has a basis given by $\tau_{S^n}$ and $\eta \iota$, where $\iota \in \pi_{n-1} \mathrm{BSO}(n) \cong \Z$ denotes a generator. Finally, we set up several invariants of $n$-spaces and definitions which will be useful later:

\begin{itemize}
\item If $n$ is even, let $\sig$ denote the signature of the symmetric bilinear form on $H$.
\item If $n \equiv 0 \mod 4$, the composition
  $ H \xrightarrow{\alpha} \pi_{n} \mathrm{BSO}(n) \to \pi_n \mathrm{BSO} \cong \Z $
  is a linear map by (\ref{eq:wall2}) and corresponds to some element $\chi \in H$ via the unimodular bilinear form.
\item If $n \equiv 1, 2 \mod 8$, the same procedure determines an element $\chi \in H/2$.
\item If $n \equiv 0,2,4\mod 8$, the self-intersection number of $\chi$ is an integer $\chi^2$ (well-defined modulo $4$ in the $n \equiv 2 \mod 8$ case).
\item If $n \equiv 1 \mod 8$, we let $\varphi$ denote the map $ \pi_n \mathrm{BSO}(n) \to \Z/2 $ with kernel $\eta \iota$. \footnote{In \cite{Wall62}, Wall did not fix a specific choice of $\varphi$, merely asking that the direct sum of $\varphi$ with the stabilization map be an isomorphism. This leads to an ambiguity in his definition of $\Phi$ when $n \equiv 1 \mod 8$. Here we are careful to fix this specific choice so as to make \Cref{thm:geography} correct and unambiguous when $n \equiv 1 \mod 8$.}
\item If $n \equiv 1 \mod 8$, the element $\varphi(\alpha(\chi))$ gives an element in $\Z/2$ which we denote $\varphi(\chi)$.
\item If $n$ is odd, then since $\tau_{S^n}$ is sent to a generator under the map $\alpha$ ($\alpha \circ \varphi$ in the $n \equiv 1 \mod 8$ case), this map determines a quadratic refinement of the mod $2$ reduction of the intersection pairing. We let $\Phi$ denote the Arf--Kervaire invariant of this quadratic form.
\end{itemize}

Our notation for these invariants follows that in \cite{Wall62}, with the exception of writing $\sig$ for the signature instead of $\tau$. Wall shows that each of the invariants $\sig, \chi^2, \Phi, \varphi(\chi)$ descends to a linear map from $A^{\langle n-1 \rangle} _{2n}$. He has further computed the groups $A^{\langle n-1 \rangle} _{2n}$ in terms of these invariants.

\begin{prop} [{\cite{Wall62,Wall67}}] \label{prop:A-comp}
  The values of $A^{\langle n-1 \rangle} _{2n}$ are given the following table, along with a choice of basis in terms of the above invariants:
  \begin{center}
    \renewcommand{\arraystretch}{1.4}
    \begin{tabular}{|c||c|c|c|c|c|c|c|c|} \hline
      $n \pmod 8$ & $0$ & $ 1$ & $ 2$ & $ 3$ & $ 4$ & $ 5$ & $ 6$ & $ 7$ \\\hline
      $A^{\langle n-1 \rangle} _{2n}$ & $\Z \oplus \Z$ & $\Z/2 \oplus \Z/2$ & $\Z \oplus \Z/2$ & $\Z/2$ & $\Z \oplus \Z$ & $\Z/2$ & $\Z$ & $\Z/2$ \\\hline
      $\mathrm{Basis}$ & $(\frac{\sig}{8}, \frac{\chi^2}{2})$ & $(\Phi,\varphi(\chi))$ & $(\frac{\sig}{8}, \frac{\chi^2}{2})$ & $\Phi$ & $(\frac{\sig}{8}, \frac{\chi^2}{2})$ & $\Phi$ & $\frac{\sig}{8}$ & $\Phi$ \\\hline
    \end{tabular}
    \renewcommand{\arraystretch}{1.0}
  \end{center}
\end{prop}

\subsection{The main theorem}
We are now ready to state and prove the main theorem of this paper, assuming \Cref{thm:main} as input.
We first recall some useful quantities computed by Krannich and Reinhold. 

\begin{dfn} \label{dfn:manynumbers}
Let $n>2$ denote a positive integer.  Following \cite{KrannichCharacteristic}, we let:
\begin{itemize}
\item $B_{2n}$ denote the $(2n)^{\mathrm{th}}$ Bernoulli number.
\item $j_{n}$ and $k_n$ denote the denominator and numerator, respectively, of the absolute value of $\frac{B_{2n}}{4n}$ when written in lowest terms.
\item $a_n$ denote $1$ if $n$ is even and $2$ if $n$ is odd.
\item $\sigma_n$ denote the integer $\sigma_n=a_n 2^{2n+1}(2^{2n-1}-1) k_n $.
\item $c_n$ and $d_n$ denote integers such that $c_n k_n + d_n j_n = 1 $.
\end{itemize}
%If $m=2k>4$ is an even integer, we follow \cite[Lemma 2.7]{KrannichCharacteristic} and let $s(Q)_{2k}$ denote the integer
Finally, we let $s(Q)_{2n}$ denote the following integer:
\[ s(Q)_{2n} = \frac{-1}{8 j^2_{n}} 
\left( \sigma_n^2 + a_n^2 \sigma_{2n} k_n \right)
\left( c_{2n} k_n + 2(-1)^{n} d_{2n} j_n \right).
\]
\end{dfn}

\begin{rmk}
  The integers and $c_n$ and $d_n$, and therefore $s(Q)_{2n}$, are not well-defined. Nevertheless, $s(Q)_{2n}$ is well-defined modulo $\frac{\sigma_{2n}}{8}$.\todo{Have not actually checked this, but things would be fucked if it were false.}
  Since we will only use the value of $s(Q)_{2n}$ modulo $\frac{\sigma_{2n}}{8}$ in this section, this will not present a problem for us.
  %The significance of $\frac{\sigma_{n/2}}{8}$ is that it is equal to $\abs{\bp_{2n}}$ when $n$ is even.
  % in \Cref{sec:class}, the statements there are well-defined.

  % On the other hand, the statement of \Cref{prop:lift-if} depends on a choice of integer $s(Q)_{6}$, and is true for all possible such choices.
\end{rmk}

\begin{thm}\label{thm:geography}
  Suppose that $ n \geq 3 $.
  With the exception of finitely many $n$, an $n$-space $(H, H \otimes H \to \Z, \alpha)$ is realized by a smooth, closed, oriented, $(n-1)$-connected $2n$-manifold if and only if the following conditions hold:
  \begin{enumerate}
    \item If $n \equiv 0 \mod 4$, then
      $ \sig + 4 s(Q)_{n/2} \chi^2 \equiv 0 \mod \sigma_{n/2} $.
    \item If $n \equiv 2 \mod 4$, then
      $\sig \equiv 0 \mod \sigma_{n/2} $.
    \item If $n \equiv 1 \mod 2$, then $\Phi = 0$.
  \end{enumerate}
  The full list of exceptions is as follows:
  \begin{itemize}
  \item If $n = 3,7,15,31$, then every $n$-space is realizable.
  \item If $n = 63$, then every $n$-space is realizable if there exists a closed smooth manifold of Kervaire invariant one in dimension $126$. Otherwise, an $n$-space is realizable if and only if $\Phi = 0$.
  \item If $n = 4$ or $8$, then instead of condition (1) we require $\sig - \chi^2 \equiv 0 \mod \sigma_{n/2}$.
  \item If $n = 9$, we require condition (3) and demand further that $\varphi(\chi) = 0$.
  \item If $n = 12$, we require condition (1) and demand further that $\chi^2 \equiv 0 \mod 4$.
  \end{itemize}
\end{thm}

\begin{rmk}
  The cases $3\leq n \leq 8$ were proven in \cite{Wall62}, so we will assume $n \geq 9$ in the following.
\end{rmk}

%% \begin{rmk}
%%   In \Cref{thm:geography}, we used the quantities $s(Q)_{n/2}$ and $\sigma_{n/2}$, whose definitions will be given in \Cref{dfn:manynumbers}. The significance of $\frac{\sigma_{n/2}}{8}$ is that it is equal to $\abs{\bp_{2n}}$ when $n$ is even.
%% \end{rmk}

Almost all of the ingredients in this theorem are already in the literature, and much of what we do here consists merely of their collation. What is original to this paper are the new cases of \Cref{thm:main}, as well as the special care given to the $n \equiv 1 \mod 8$ case, which we have not seen spelled out elsewhere.

The proof of \Cref{thm:geography} will follow immediately from \Cref{prop:A-comp} and \Cref{lem:partial} below, which computes the boundaries of several specific classes. Recall the Kervaire--Milnor exact sequence
\[ 0 \to \bp_{2n} \to \Theta_{2n-1} \to \coker(J)_{2n-1}. \]

Given an element $\Sigma \in \Theta_{2n-1}$, we will let $[\Sigma] \in \coker(J)_{2n-1}$ denote its image under the map in the above exact sequence.

When $n$ is even this sequence is short exact, and Brumfiel constructed a preferred splitting \cite{Brumfiel}:
\[ \Theta_{2n-1} \cong \bp_{2n} \oplus \coker(J)_{2n-1}. \]

\begin{lem} \label{lem:partial}\ 
\begin{enumerate}
  \item For $n$ even, let $P \in A^{\langle n-1 \rangle} _{2n}$ denote the element with $\frac{\sig}{8} = 1$ and $\frac{\chi^2}{2} = 0$ if $n \equiv 0,2,4 \mod 8$. As noted in \cite[Section 3.2.2]{KrannichMappingClass}, we may choose $P$ to be Milnor's $E_8$-plumbing. The homotopy sphere $\partial(P) \in \Theta_{2n-1}$ is a generator of $\bp_{2n}$, which is a cyclic group of order $\frac{\sigma_{n/2}}{8}$.
  \item For $n \equiv 0 \mod 4$, let $Q \in A^{\langle n-1 \rangle} _{2n}$ denote the element with $( \frac{\sig}{8}, \frac{\chi^2}{2}) = (0,1)$. For $n \neq 12$, $\partial(Q)$ is $s(Q)_{n/2} \cdot \partial(P)$. For $n = 12$, $\partial(Q) $ is $s(Q)_6 \cdot \partial(P) + \eta^3 \bar{\kappa}$, where $\eta^3 \kappabar \in \coker(J)_{23}$ is viewed as an element of $\Theta_{23}$ via Brumfiel's splitting. The class $\eta^3 \kappabar$ is simple $2$-torsion.
  \item For $n \equiv 2 \mod 4$, let $L \in A^{\langle n-1 \rangle} _{2n}$ denote the element with $( \frac{\sig}{8}, \frac{\chi^2}{2}) = (0,1)$. Then $\partial (L) = 0$.
  \item For $n$ odd, let $K \in A^{\langle n-1 \rangle} _{2n}$ denote the element with $\Phi(K) = 1$ (and $\varphi(\chi) = 0$ if $n \equiv 1 \mod 8$). We may choose $K$ to be the Kervaire plumbing. The homotopy sphere $\partial (K) \in \Theta_{2n-1}$ is a generator of $\bp_{2n}$. This group is zero for $n =1,3,7,15,31$ and possibly $63$, depending on the outcome of the $126$-dimensional Kervaire invariant one problem. It is $\Z/2$ otherwise.
  \item For $n \equiv 1 \mod 8$, let $R \in A^{\langle n-1 \rangle} _{2n}$ denote the element with $(\Phi, \varphi(\chi)) = (0,1)$. For $n \neq 9$, we have $\partial (R) = 0$. For $n = 9$, we have $[\partial (R)] = \eta \eta_4 \in \coker(J)_{17}$, which is simple $2$-torsion
\end{enumerate}
\end{lem}

\begin{proof}[Proof of \Cref{lem:partial}(1)-(4)]\ 
  
  {\bf The element P.} The boundary of Milnor's $E_8$-plumbing is well-known to be a generator of $\bp_{2n}$ when $n$ is even (use e.g. \cite[Lemma 3.5(2)]{Levine} and the fact the signature of the Milnor plumbing is equal to $8$). The fact that for even $n$ the group $\bp_{2n}$ is a cyclic of order $\frac{\sigma_{n/2}}{8}$ follows from \cite[Corollary 3.20]{Levine}.

  {\bf The element Q.} %In order to understand $\partial Q$ we begin by noting that Brumfiel has shown that there is a canonical splitting $ \Theta_{2n-1} \cong \bp_{2n} \oplus \coker(J)_{2n-1} $ \cite{Brumfiel}.
  We will describe $\partial(Q)$ in terms of Brumfiel's splitting.
  On the one hand, Krannich--Reinhold \cite[Lemma 2.7]{KrannichCharacteristic}, building on work of Stolz \cite{StolzbP}, have computed the $\bp_{2n}$-component of $\partial(Q)$ to be $s(Q)_{n/2} \cdot \partial(P)$, for the explicit quantity $s(Q)_{n/2}$ defined in \Cref{dfn:manynumbers}. On the other hand, the computation of the $\coker(J)_{2n-1}$-component of $\partial(Q)$ follows immediately from \Cref{thm:main} and the fact that $\partial (P)$ lies in $\bp_{2n}$. In particular, the $\coker(J)_{2n-1}$-component of $\partial(Q)$ is zero if $n \neq 12$ and is $\eta^3 \kappabar$ for $n=12$.
  %Note that $\eta^3 \kappabar$ is simple $2$-torsion.

{\bf The element L.} It is a result of Schultz that $\partial(L) = 0 \in \Theta_{2n-2}$ \cite[Corollary 3.2 and Theorem 3.4(iii)]{Schultz}.

{\bf The element K.} The boundary of the Kervaire plumbing, $\partial (K)$, is the Kervaire sphere, which generates $\bp_{2n}$ by \cite{KervaireMilnor}. This case of \Cref{thm:geography} now follows from the fact that, if $n$ is odd, $\bp_{2n} \cong 0$ if a smooth closed $2n$-manifold of Kervaire invariant one exists, and is isomorphic to $\Z / 2 \Z$ otherwise. By \cite{BPKervaire,MahTanDiff,BrowderKervaire,BJMKervaire,HHR}, such a manifold exists if $n = 1, 3, 7, 15, 31$ and possibly if $n = 63$.
\end{proof}

Computing $\partial(R)$ takes more work.
It follows from \Cref{thm:main} that for $n \neq 9$, $[\partial(R)] = 0 \in \coker(J)_{2n-1}$, but to show that the condition for realizability of an $n$-space is $\Phi = 0$ and not $\Phi + \varphi(\chi) = 0$, it is necessary to prove that $\partial(R) = 0 \in \Theta_{2n-1}$. To do this, we recall an argument of Schultz \cite{Schultz} which reduces it to the case of $n \equiv 0 \mod 8$. 

Recall Bredon's pairing 
\[- \cdot - : \Theta_n \times \pi_{n+k} (S^n) \to \Theta_{n+k}, \]
which is bilinear and defined for $n \geq 5$ \cite{Bredon}. Roitberg has shown that this pairing is that induced by the composition action on $\Theta_n \cong \pi_n PL/O$, and used this to show that the restriction of this pairing to $\bp_{n+1}$ is zero for $k \geq 1$  \cite[Theorem B]{Roitberg}.

\begin{lem} \label{prop:careful-Schultz}
  Suppose that $n > 9$ is congruent to $1$ modulo $8$.
  Then
  $ \partial(R) = \partial(Q) \cdot \eta^2,$ where we have used $\cdot$ for Bredon's pairing.
\end{lem}

\begin{proof}
  Let $\alpha \in \pi_{n-1} \BSO(n-1)$ denote a generator of the image of $\pi_{n-1} \BSO(n-2)$ in $\pi_{n-1} \BSO(n-1)$. It maps under the stabilization map $i_* : \pi_{n-1} \BSO(n-1) \to \pi_{n-1} \BSO(n) \cong \Z$ to a generator.

  As in \cite[Section 3.2.2]{KrannichMappingClass}, $Q$ may be chosen to be the plumbing of two copies of the $(n-1)$-dimensional linear disk bundle over $S^{n-1}$ corresponding to $\alpha$. Similarly, it is not hard to see that $R$ may be constructed by plumbing together two copies of the $n$-dimensional linear disk bundle over $S^{n}$ corresponding to $i_* \eta \alpha$.
  It then follows from \cite[Theorem 2.5]{Schultz} that
  \[\partial(R) = \partial(Q) \cdot \eta^2,\]
  as desired.
\end{proof}

\begin{rmk}
  This is essentially \cite[Theorem 3.1]{Schultz}, except that Schultz does not specify the definition of $\Phi$ that he is using.
\end{rmk}

\begin{proof}[Proof of \Cref{lem:partial}(5)]
  Suppose that $n > 9$.
  We have already seen above that $\partial(Q) \in \bp_{2n-2}$. (Note that $n-1 \equiv 0 \mod 8$ and in particular cannot be equal to $12$.)
  Then \Cref{prop:careful-Schultz} asserts that $\partial(R) = \partial(Q) \cdot \eta^2$, which is equal to $0$ because Bredon's pairing restricts to zero on $\bp$. 

  On the other hand, at $n=9$, \Cref{thm:main}, \Cref{prop:A-comp} and the computation of $\partial(K)$ imply that we must have $[\partial R] = \eta \eta_{4} \in \coker(J)_{17}$.
\end{proof}

\section{The Witten genus of $11$-connected $24$-manifolds} \label{sec:witten}
In this section, we will show that the $23$-dimensional case of \Cref{thm:main} is exceptional, disproving the conjecture of Galatius and Randal-Williams.
In particular we prove the following theorem:

\begin{thm} \label{thm:counterex}
  %The kernel of the $2$-local unit map
  %\[\pi_{23} \Ss_{(2)} \to \pi_{23} \MO\langle12\rangle_{(2)}\]
  %is not generated by the image of $J$.
%
  The image of the composition
  \[A^{\langle 11 \rangle} _{24} \xrightarrow{\partial} \Theta_{23} \to \coker(J)_{23} \to \left(\coker(J)_{23}\right)_{(2)}\]
  is not $0$. In particular, there is an $11$-connected $24$-manifold whose boundary is a homotopy sphere which does not bound a parallelizable manifold.
\end{thm}

%To prove \Cref{thm:counterex}, we will use the exact sequence of \Cref{thm:...}:
  %\[\Theta_{2n} \to \Omega^{\langle n-1 \rangle} _{2n} \to A^{\langle n-1 \rangle} _{2n} \xrightarrow{\partial} \Theta_{2n-1}.\]
  %Using this exact sequence and \Cref{...}, we know that, if the image of the composite
  %\[A^{\langle 11 \rangle} _{24} \xrightarrow{\partial} \Theta_{23} \to \coker(J)_{23}\]
  %is zero, the image of
  %\[\Omega^{\langle 11 \rangle} _{24} \to A^{\langle 11 \rangle} _{24}\]
  %will contain the element $Q - s(Q)_{6} P$ in the notation of \Cref{...}.

\begin{rmk}
  In Sections \ref{sec:inch} and \ref{sec:3-3}, we refine this result, showing that the image of the composition
  \[A^{\langle 11 \rangle} _{24} \xrightarrow{\partial} \Theta_{23} \to \coker(J)_{23}\]
  is $\{0, \eta^3 \kappabar\} \subset \coker(J)_{23}$.
\end{rmk}

The proof of \Cref{thm:counterex} is a relatively straightforward application of the Ando--Hopkins--Rezk refinement \cite{AHR} of the Witten genus and will be accomplished in two steps.
First, we will collect results of Hirzebruch, Hopkins--Mahowald and Ando--Hopkins--Rezk that provide a divisibility condition on the Pontryagin numbers of a closed $11$-connected $24$-manifold. Second, we will work through the implications of this restriction for the relevant examples.

  %% To prove \Cref{thm:counterex}, we begin with the following proposition, which follows immediately from the results of \Cref{sec:class}:

  %% \begin{prop} \label{prop:lift-if}
  %%   Let $P \in A^{\langle 11 \rangle} _{24}$ and $Q \in A^{\langle 11 \rangle} _{24}$ represent the classes with $(\frac{\sig}{8}, \frac{\chi^2}{2}) = (1,0)$ and $(\frac{\sig}{8}, \frac{\chi^2}{2}) = (0,1)$, respectively.
  %%   Moreover, let $[\Sigma_Q]$ equal the image of $Q$ in $\coker(J)_{23}$ under the composite
  %%   \[A^{\langle 11 \rangle} _{24} \xrightarrow{\partial} \Theta_{23} \to \coker(J)_{23},\]
  %%   and let $\ord([\Sigma_Q])$ denote its order.
  %%   Then, for any choice of $s(Q)_6 \in \Z$ as in \Cref{dfn:manynumbers}, the class
  %%   \[\ord([\Sigma_Q]) (Q-s(Q)_6 P) \in A^{\langle 11 \rangle} _{24}\]
  %%   lifts to $\Omega^{\langle 11 \rangle} _{24}$.
  %% \end{prop}

  %% To prove \Cref{thm:counterex}, it will therefore suffice to show that no odd multiple of $Q - s(Q)_6 P$ lies in the image of
  %% $\Omega^{\langle 11 \rangle} _{24} \to A^{\langle 11 \rangle} _{24}.$
  %% We will do this in two steps.
  %% In \Cref{sec:cond}, we will use the Ando--Hopkins--Rezk refinement of the Witten genus \cite{AHR} to establish a nontrivial condition satisfied by the Pontryagin numbers of any closed $11$-connected $24$-manifold.
  %% In \Cref{sec:Wittenapp}, we will conclude by showing that the Pontryagin numbers of an odd multiple of $Q - s(Q)_6 P$ do not satisfy this condition.
  
%% \subsection{A condition on the Pontryagin numbers of a closed $11$-connected $24$-manifold.}\label{sec:cond} 
\subsection{A condition on Pontryagin numbers.}\label{sec:cond}
In order to motivate the condition on the Pontryagin numbers of an $11$-connected $24$-manifold which we need, we will begin by working through an analogous restriction for $4$-dimensional spin manifolds, which is equivalent to Rokhlin's theorem.

\begin{exm}
  The Hirzebruch signature formula tells us that the signature of the intersection form on $\mathrm{H}_2$ of an oriented $4$-manifold is given by $\frac{1}{3}p_1$, where $p_1$ is the first Pontryagin number. In the case where the $4$-manifold is spin (i.e. $w_2 = 0$), the intersection form on $\mathrm{H}_2$ is unimodular and even. Any even unimodular quadratic form has signature divisible by 8, so we may conclude that $p_1$ is divisible by 24.

  Further divisibility conditions can be obtained by a more sophisticated analysis.
  Given a spin manifold, we can consider its $\Ah$-genus, which is an integer (being the index of the Dirac operator on the spinor bundle). In this case, the $\Ah$-genus is equal to $-\frac{1}{24}p_1$, from which we immediately recover our earlier divisibility criterion. But we may go further. Indeed, since the real Clifford algebra $C\ell_4 \cong M_2 (\mathbb{H})$ contains the quaternions, the spinor representation inherits a quaternionic structure. The Dirac operator is then quaternion-linear, so its index must be even. The same conclusion can be obtained using the Atiyah--Bott--Shapiro orientation
  \[ \MSpin \to \mathrm{ko},\]
  which refines the $\Ah$-genus \cite{ABS}.
  Indeed, the composite
  \[ \pi_4 \mathrm{MSpin} \to \pi_4 \mathrm{ko} \to \pi_{4} \mathrm{ku} \cong \Z, \]
  is equal to the $\Ah$-genus, where the first map is induced by the Atiyah--Bott--Shapiro orientation and the second map is induced by tensoring up a real vector bundle to the complex numbers.
  Since the second map sends a generator to twice a generator of $\Z$, we learn that the $\Ah$-genus of a spin $4$-manifold is divisible by 2, i.e. that $p_1$ is divisible by $48$.
\end{exm}

In the case of $11$-connected $24$-manifolds, we will use the Ando--Hopkins--Rezk refinement of the Witten genus to prove the following:

\begin{prop} \label{prop:divisibility}
  If $M$ is an $11$-connected $24$-manifold, then $-1177p_3^2 - 311p_6$ is divisible by $237758976000 = 2^{12} \cdot 3^6 \cdot 5^3 \cdot 7^2 \cdot 13 $.
\end{prop}

%% We begin with a short recapitulation of the properties of the Witten genus.
%% Next, we discuss integrality properties of the value of the Witten genus.
%% Finally, we extract formulas for the Witten genus in terms of Pontryagin numbers from \cite{Hirzebruch} and \cite{HirzebruchPrize}.

The Witten genus is a cobordism invariant of string manifolds which takes values in the ring of integral modular forms, c.f. \Cref{dfn:mf}. However, not every integral modular form is the Witten genus of a string manifold. Indeed, the Witten genus factors through the coefficient ring of the spectrum of topological modular forms and it is this restriction which will provide the leverage we need to prove \Cref{prop:divisibility}.

%% \begin{rmk}
%%   This is analogous to the classical fact that the $\widehat{A}$-genus of a $n$-dimensional spin manifold is even if $n \equiv 4 \mod 8$.
%% \end{rmk}
%
%\todo{Historical discussion: Rokhlin's theorem, divisibility of Ahat genus, etc. Note that there exists this paper of Chen and Han.}
%
%We will leverage this restriction on the Witten genus into our condition on the Pontryagin numbers of closed $11$-connected $24$-manifolds.

\begin{dfn} \label{dfn:mf}
  Let $\MF_n$ denote the group of integral modular forms of weight $n$. The direct sum of these groups, $\MF_* = \bigoplus_{n \geq 0} \MF_n$, is a graded ring, and has explicit generators and relations
  \[\MF_* \cong \Z[E_4, E_6, \Delta] / (E_4 ^3 - E_6 ^2 - 1728 \Delta),\]
  where $E_4$ and $E_6$ are the weight $4$ and $6$ normalized Eisenstein series, respectively, and $\Delta$ is the discriminant.
\end{dfn}
%
%Using $MF_* \otimes \QQ$ as the ring of coefficient we may consider the following power series
  %\[ \frac{z}{\sigma_L(z)} = \exp \left( \sum_{k \geq 2} \frac{2G_{2k} z^{2k}}{(2k)!} \right), \]
  %where $\sigma_L$ is the Weiertrass sigma fu
  %The genus associated to this power series is called the Witten genus.

Let $\Omega_* ^{\Str}$ denote the cobordism ring of string manifolds.
The Witten genus is a ring map
\[\phi_W : \Omega_{2*} ^{\Str} \to \MF_{*}.\]

In fact, by work as Ando, Hopkins and Rezk \cite{AHR}, the Witten genus can be refined to a map of $\mathbb{E}_\infty$-ring spectra as follows. Let $\MStr$ denote the Thom spectrum of the canonical map 
$ \BStr = \tau_{\geq 8} \BO \to \BO $.
Then there is a canonical isomorphism $\pi_* \MStr \cong \Omega_* ^{\Str}$.
Let $\tmf$ denote the connective spectrum of topological modular forms \cite{HMtmf, Behtmf, ECII}. This is an $\mathbb{E}_\infty$-ring spectrum which comes equipped with a ring map $\pi_{2*} \tmf \to \MF_{*}$. Ando, Hopkins and Rezk \cite{AHR} proved the Witten genus lifts to a map of $\mathbb{E}_\infty$-ring spectra.

\begin{thm}[\cite{AHR}]
    There is a map of $\mathbb{E}_\infty$-rings $\MStr \to \tmf$ such that the induced map
    \[\Omega^{\Str} _{2*} \cong \pi_{2*} \MStr \to \pi_{2*} \tmf \to \MF_{*}\]
    is the Witten genus $\phi_W$.
\end{thm}

Hopkins and Mahowald have determined the image of the map $\pi_{2*} \tmf \to \MF_*$:

\begin{thm} [{\cite[Proposition 4.6]{HopkinsICM}}]
    The image of the map $\pi_{2*} \tmf \to \MF_*$ has a basis given by the monomials
    \[a_{i,j,k} E_4 ^i E_6 ^j \Delta^k \,\,\, i,k \geq 0, j=0,1\]
    where
    \begin{align*}
        a_{i,j,k} = \begin{cases} 1 & i>0, j=0 \\ 2 & j=1 \\ 24/\gcd(24,k) & i,j=0. \end{cases}
    \end{align*}
\end{thm}

Specializing to the case of $24$-dimensional manifolds, we obtain the following result.

\begin{cor}
    The image of the Witten genus $\Omega^{\Str} _{24} \to \MF_{12}$ lies in the subspace of $\MF_{12} = \Z\{E_4 ^3, \Delta\}$ spanned by $E_4 ^3$ and $24\Delta$.
\end{cor}

%\begin{rmk}
    %In fact, it is known that the map $\MStr \to \tmf$ is surjective on homotopy groups, so that the image of the Witten genus is \textit{equal} to the image of the map $\pi_* \tmf \to \MF_{*/2}$.
%\end{rmk}
%
%
%This is what we will need to derive the desired contradiction to the statement that $[\Sigma_Q] = 0$. We make this more explicit.
%
%\begin{fact}
    %The cokernel of the map $\pi_{24} \tmf \to \MF_{12}$ is a $\Z/24\Z$ generated by $\Delta$.
%\end{fact}
%

\begin{rmk}
  At this point one could obtain the desired divisibility criterion on the Pontryagin numbers as follows.
  The Witten genus is determined by the following characteristic series:
  \[\exp \left( \sum_{k \geq 1} 2G_{2k} \frac{z^{2k}}{(2k)!} \right), \]
  where $G_{2k}$ is the weight $2k$ Eisenstein series. From this, one can extract a polynomial in the Pontryagin numbers which computes the Witten genus for $24$-manifolds. Plugging in the fact that all terms except $p_3^2$ and $p_6$ are zero due to the $11$-connectedness assumption would prove the proposition.

  Instead of doing this, we will cite results of Hirzebruch and Hopkins--Mahowald, since they have already analyzed the $24$-dimensional case in connection with the Hirzebruch prize manifold.
\end{rmk}

%% Hirzebruch \cite{Hirzebruch} has shown how to express the Witten genus of a $24$-dimensional string manifold in terms of more classical invariants.

\begin{lem}[{\cite[Example on pg. 85-86]{Hirzebruch}}]
    Let $M$ denote a $24$-dimensional string manifold. Then
    \[\phi_W (M) = \Ah(M) \overline{\Delta} + \Ah(M, T_\CC)\Delta,\]
    where $\overline{\Delta} = E_4 ^3 - 744 \Delta$, $\Ah(M)$ is the $\Ah$-genus of $M$, and $\Ah(M, T_\CC)$ is the twisted $\Ah$-genus of $M$, where the twisting is by the complexified tangent bundle $T_\CC$.
\end{lem}

Since $744$ is divisible by $24$, we have the following corollary:

\begin{cor} \label{cor:24}
    Let $M$ denote a $24$-dimensional string manifold. Then $\Ah(M,T_\CC)$ is divisible by $24$.
\end{cor}
%
%\begin{ntn}
    %When this makes sense \todo{Expand on when twisted Ahat makes sense.}, we let $\Ah(M)$ denote the $\Ah$-genus of $M$. Moreover, we let $\Ah(M, T_\CC)$ denote the twisted $\Ah$-genus of $M$, where the twisting is by the complexified tangent bundle $T_\CC$.
%\end{ntn}
%
%\begin{fact}
    %Let $M$ denote a $24$-dimensional string manifold. Then
    %\[\phi_W (M) = \Ah(M) \overline{\Delta} + \Ah(M, T_\CC)\Delta,\]
    %where $\overline{\Delta} = E_4 ^3 - 744 \Delta$.
%\end{fact}
%
%\begin{thm}[Hirzebruch prize problem]
    %$\overline{\Delta}$ is in the image of the Witten genus.
%\end{thm}
%
%\begin{cor}
%    An expression of the form $a\overline{\Delta} + b\Delta$ for $a,b \in \Z$ lies in the image of the Witten genus if and only if $b \in 24 \Z$.
%\end{cor}
%

Finally, we recall the formula for $\Ah(M,T_\CC)$ in terms of Pontryagin numbers, conveniently provided by Hopkins and Mahowald \cite[pg. 98]{HirzebruchPrize}, specialized to the case of $11$-connected $24$-manifolds:

\begin{prop}[{\cite[pg. 98]{HirzebruchPrize}}] \label{prop:twist-Ah-p}
    Let $M$ denote a smooth, closed, oriented $11$-connected $24$-manifold. Then
    \[\Ah(M, T_\CC) = \frac{-1177 p_3 ^2 - 311 p_6}{9906624000}.\]
\end{prop}

Combining this with \Cref{cor:24}, we obtain \Cref{prop:divisibility}.

%% explicit divisibility criterion on the Pontryagin numbers of an $11$-connected $24$-dimensional manifold, since any such manifold admits a string structure.

\subsection{Application}\label{sec:Wittenapp}
To prove \Cref{thm:counterex}, we begin by using the results of \Cref{sec:class} to construct an element of $\Omega^{\langle 11 \rangle} _{24}$. Let $P \in A^{\langle 11 \rangle} _{24}$ and $Q \in A^{\langle 11 \rangle} _{24}$ represent the classes with $(\frac{\sig}{8}, \frac{\chi^2}{2}) = (1,0)$ and $(\frac{\sig}{8}, \frac{\chi^2}{2}) = (0,1)$, respectively.
Moreover, let $[\Sigma_Q]$ equal the image of $Q$ in $\coker(J)_{23}$ under the composite
\[A^{\langle 11 \rangle} _{24} \xrightarrow{\partial} \Theta_{23} \to \coker(J)_{23},\]
and let $\ord([\Sigma_Q])$ denote its order.
Then, for any choice of $s(Q)_6 \in \Z$ as in \Cref{dfn:manynumbers}, the class
\[\ord([\Sigma_Q]) (Q-s(Q)_6 P) \in A^{\langle 11 \rangle} _{24}\]
lifts to $\Omega^{\langle 11 \rangle} _{24}$. \footnote{This statement depends on a choice of integer $s(Q)_{6}$ and is true for all possible such choices.}

Assuming the class $N(Q - s(Q)_6 P)$ lifts to $\Omega^{\langle 11 \rangle} _{24}$ for some integer $N$ we can compute its Pontryagin numbers and check whether they are compatible with \Cref{prop:divisibility}. This will allow us to conclude that $\ord([\Sigma_Q])$ is even.

Following \Cref{dfn:manynumbers} step-by-step, we now compute (a choice of) $s(Q)_{6} \in \Z$:
\begin{itemize}
    \item $B_{6} = \frac{1}{42}$, $B_{12} = -\frac{691}{2730}$.
    \item $j_3 = \denom(\frac{\abs{B_6}}{4\cdot 3}) = \frac{1}{504}$.
    \item $k_3 = \num(\frac{\abs{B_6}}{4\cdot 3}) = 1$, $k_6 = \num(\frac{\abs{B_{12}}}{4 \cdot 6}) = 691$.
    \item $a_3 = 2$, $a_6 = 1$.
    \item $\sigma_3 = a_3 \cdot 2^{7} (2^5 -1) k_3 = 2 \cdot 2^7 (2^5-1)$.
    \item $\sigma_6 = a_6 \cdot 2^{13} (2^{11} -1) k_6 = 1 \cdot 2^{13} (2^{11}-1) 691$.
    \item We choose $c_6 = -18869$ and $d_6 = 199$.
\end{itemize}

    Finally, we find that
    \begin{align*}
      %s(Q)_6 &= \frac{1}{8 \cdot 504^2}\left( (2 \cdot 2^{7} (2^{5}-1))^2 + 2^2 \cdot (2^{13} (2^{11}-1)691)\cdot 1 \left(-18869+2(-1)199\cdot504 \right) \right) \\
      s(Q)_6 &= \frac{-\left( (2 \cdot 2^{7} (2^{5}-1))^2 + 2^2 \cdot (2^{13} (2^{11}-1)691)\cdot 1 \right) \left(-18869\cdot 1+2 \cdot (-1)\cdot 199\cdot504 \right)}{8 \cdot 504^2} \\
      &= -5005553600.
    \end{align*}

%% \begin{cor}
%%   %If the composite
%%   %\[A^{\langle 11 \rangle} _{24} \xrightarrow{\partial} \Theta_{23} \to \coker(J)_{23}\]
%%   %is zero, then 
%%   The class $\ord([\Sigma_Q])(Q + 5005553600 \cdot P) \in A^{\langle 11 \rangle} _{24}$ lifts to $\Omega^{\langle 11 \rangle} _{24}$.
%% \end{cor}

%% We now need to compute the Pontryagin numbers of this bordism class of highly connected manifolds, and show that the results of \Cref{sec:cond} imply that $\ord([\Sigma_Q])$ must be divisible by $2$.
%Note that, since we are working with $11$-connected manifolds, the only possible nonzero Pontryagin numbers are $p_3^2$ and $p_6$. We state the answer below.
%
%\begin{prop}
%    The Pontryagin numbers of any lift of $\ord([\Sigma_Q]) (Q + 5005553600 \cdot P)$ to $\Omega^{\langle 11 \rangle} _{24}$ are
%    \[p_3 ^2 = 115200 \cdot \ord([\Sigma_Q])\]
%    and
%    \[p_6 = -9038281766400 \cdot \ord([\Sigma_Q]).\]
%\end{prop}
%

In order to compute the Pontryagin numbers of $N(Q-s(Q)_6 P)$, we make use of the formulas provided by Krannich and Reinhold. 

\begin{prop}[{\cite[Proposition 2.13]{KrannichCharacteristic}}]
  Given $n \geq 3$ and a choice of $s(Q)_{2n} \in \Z$, the Pontryagin numbers of a lift of $N (Q - s(Q)_{2n} P)$ to $\Omega^{\langle 4n-1 \rangle} _{8n}$ are given by
  \[p_n ^2 = 2 N a_n ^2 (2n-1)!^2, \]
  \[p_{2n} = N a_n ^2 \left((2n-1)!^2+(4n-1)! j_{2n} \frac{\abs{B_{2n}}}{4n}\left(c_{2n} \frac{\abs{B_{2n}}}{4n}+2d_{2n}(-1)^n \right)\right),\]
  for the choices of $c_{2n}$ and $d_{2n}$ corresponding to the choice of $s(Q)_{2n}$.
\end{prop}

Plugging in the values for $n = 3$ and making sure to choose the same $c_6$ and $d_6$ that we did above, we obtain:

\begin{cor}
    The Pontryagin numbers of any lift of $N (Q + 5005553600 \cdot P)$ to $\Omega^{\langle 11 \rangle} _{24}$ are
    \[p_3 ^2 = 115200 \cdot N \qquad \text{ and } \qquad p_6 = -9038281766400 \cdot N.\]
\end{cor}

Finally, plugging these values into \Cref{prop:divisibility}, we learn that
\[ 2810905493760000 \cdot N = 2^{11} \cdot 3^6 \cdot 5^4 \cdot 7^2 \cdot 13 \cdot 4729 \cdot N \]
is divisible by
\[ 237758976000 = 2^{12} \cdot 3^6 \cdot 5^3 \cdot 7^2 \cdot 13, \]
and in particular that $N$ must be even.

%% Plugging these values into \Cref{prop:twist-Ah-p}, we obtain for $M$ a lift of
%% \[\ord([\Sigma_Q]) (Q + 5005553600 \cdot P)\]
%% to $\Omega^{\langle 11 \rangle} _{24}$:
%% \[\Ah(M,T_\CC) = 283740 \cdot \ord([\Sigma_Q]) = 2^2 \cdot 3 \cdot 5 \cdot 4729 \cdot \ord([\Sigma_Q]).\]
%% Since $24$ must divide $\Ah(M, T_\CC)$ by \Cref{cor:24}, we conclude that $\ord([\Sigma_Q])$ is divisible by $2$, which proves \Cref{thm:counterex}.

\section{Reduction to homotopy theory} \label{sec:homotopy}
In this section, we use standard arguments to reduce the proof of \Cref{thm:main} to \Cref{thm:unit-ker}, a statement in stable homotopy theory.
We then further divide the proof of \Cref{thm:unit-ker} into several cases which are the subjects of sections \ref{sec:inch} through \ref{sec:medium}.
In the statement of \Cref{thm:unit-ker}, we have only included the cases not already in the literature.

\begin{thm} \label{thm:unit-ker}
  Let $\MO\langle n \rangle$ denote the Thom spectrum of the canonical vector bundle on $\tau_{\geq n} \mathrm{BO}$
  %% \[\tau_{\geq n} \mathrm{BO} \to \mathrm{BO}.\]
  Then, the kernel of unit map 
  \[\pi_{2n-1} \Ss \to \pi_{2n-1} \mathrm{MO}\langle n \rangle \]
  may be described as follows:
  \begin{enumerate}
  \item When $n=9$, it is generated by the image of the $J$-homomorphism and $\eta \eta_4$.
  \item When $n=12$, it is generated by the image of the $J$-homomorphism and $\eta^3 \kappabar$.
  \item When $n \equiv 0 \mod 4$ and $16 \leq n \leq 124$, it is equal to the image of the $J$-homomorphism.
  \end{enumerate}
\end{thm}

\begin{proof}[Proof of \Cref{thm:main} from \Cref{thm:unit-ker}]
  By \Cref{rmk:new}, the only new cases of \Cref{thm:main} that we need to establish are:
  \begin{itemize}
    \item When $n \equiv 0 \mod 4$ and $12 \leq n \leq 124$.
    \item When $n \equiv 1 \mod 8$ and $9 \leq n \leq 129$.
  \end{itemize}

  If $n = 9$ or $n = 4m$ and $3 \leq m \leq 31$, \Cref{thm:unit-ker} implies \Cref{thm:main} by \cite[Satz 1.7]{StolzBook}.
  On the other hand, \cite[Theorem 3.1]{Schultz} reduces the $n \equiv 1 \mod 8$, $n > 9$ case of \Cref{thm:main} to the $n \equiv 0 \mod 8$, $n > 8$ case, which we have already shown.
\end{proof}

To prove \Cref{thm:unit-ker}, we first note that the kernel of the unit map
\[\pi_{2n-1} \Ss \to \pi_{2n-1} \MOn\]
is well-known to contain the image of $J$ in degrees $n-1$ and above, so the question is mostly reduced to finding an effective upper bound on the size of this kernel.
%% so we need to understand the kernel of the induced map
%% \[\coker(J)_{2n-1} \to \pi_{2n-1} \MOn.\]
%% The proof of \Cref{thm:unit-ker} splits into several cases. Before we enumerate them, we recall what 
Our proof of this upper bound in the $n \equiv 0 \mod 4$ case extends the methods and results of \cite{Boundaries}, so we begin by reviewing the main points of the arguments therein.

\begin{thm}[{\cite[Lemma 6.9 and Theorem 7.1]{Boundaries}}]\label{thm:known}
  The kernel of the $p$-localized unit map
  \[ u_{8m-1} : \pi_{8m-1} \Ss_{(p)} \to \pi_{8m-1} \MOfn_{(p)} \]
  is generated by the image of $J$ if any of the following conditions are met:
  \begin{itemize}
    \item $p \geq 5$.
    \item $m \geq 32$ and $p = 3$.
    \item $m \geq 17$ and $p = 2$.
  \end{itemize}
\end{thm}

%% It follows that the remaining cases in \Cref{thm:unit-ker} are:
%% \begin{itemize}
%%   \item $p = 2$, $n = 4m$ and $3 \leq m \leq 31$.
%%   \item $p = 3$, $n = 4m$ and $3 \leq m \leq 13$.
%%   \item $p = 2$ and $n = 9$.
%% \end{itemize}

%% These divide by method of proof into five separate cases:

%% \begin{enumerate}
%%   \item The case when $n = 4m$ and $(p,m) \neq (2,3), (2,4), (2,6), (3,3), (3, 5)$.
%%   \item The case when $n = 4m$ and $(p,m) = (2,4), (2,6), (3,5)$.
%%   \item The case when $n = 4m$ and $(p,m) = (2,3)$.
%%   \item The case when $n = 4m$ and $(p,m) = (3,3)$.
%%   \item The case when $(p,n) = (2,9)$.
%% \end{enumerate}

\begin{table}
    \begin{center}
      \scalebox{0.9}{
      \begin{tabular}{|c||c|c|c|c|c|c|c|c|c|c|c|c|c|c|} \hline
        $m$  & 3 & 4 & 5 & 6 & 7 & 8 & 9 & 10 & 11 & 12 & 13 & 14 & 15 & 16 \\\hline
        $8m-1$  & 23 & 31 & 39 & 47 & 55 & 63 & 71 & 79 & 87 & 95 & 103 & 111 & 119 & 127 \\\hline
        $2N_2 - 1$  & 5 & 5 & 11 & 13 & 19 & 19 & 25 & 27 & 33 & 35 & 41 & 43 & 49 & 49 \\\hline
        $\Gamma_2(8m-1)$ & 9 & 5 & 9 & 13 & 0 & 9 & 15 & 14 & 15 & $\leq 19$ & $<37.9$ & $<41.3$ & $<42.7$ & $<49.1$ \\\hline
        $2N_3 - 1$  & 1 & 3 & 5 & 7 & 7 & 9 & 11 & 13 & 15 & 17 & 19 & 21 & 23 & 25 \\\hline
        $\Gamma_3(8m-1)$ & 5 & 0 & 5 & 5 & 5 & 0 & 0 & 0 & 0 & 8 & 0 & $ \leq 14.2 $ & $ \leq 14.7 $ & $ \leq 15.2 $ \\\hline
      \end{tabular}}

      \vspace{0.3cm}

      \scalebox{0.9}{
      \begin{tabular}{|c||c|c|c|c|c|c|c|c|c|c|c|c|c|c|c|} \hline
        $m$  & 17 & 18 & 19 & 20 & 21 & 22 & 23 & 24  \\\hline
        $8m-1$ & 135 & 143 & 151 & 159 & 167 & 175 & 183 & 191 \\\hline
        $2N_3 - 1$ & 27 & 29 & 31 & 33 & 33 & 35 & 37 & 39  \\\hline
        $\Gamma_3(8n-1)$ & $ \leq 15.7 $ & $ \leq 16.2 $ & $ \leq 16.7 $ & $ \leq 17.2 $ & $ \leq 17.7 $ & $ \leq 18.2 $ & $ \leq 18.7 $ & $ \leq 19.2 $ \\\hline
      \end{tabular}}
      
      \vspace{0.3cm}

      \scalebox{0.9}{
      \begin{tabular}{|c||c|c|c|c|c|c|c|c|c|c|c|c|c|c|c|} \hline
        $m$  & 25 & 26 & 27 & 28 & 29 & 30 & 31 \\\hline
        $8m-1$ & 199 & 207 & 215 & 223 & 231 & 239 & 247 \\\hline
        $2N_3 - 1$ & 41 & 43 & 45 & 47 & 49 & 51 & 53 \\\hline
        $\Gamma_3(8n-1)$ & $ \leq 19.7 $ & $ \leq 20.2 $ & $ \leq 20.7 $ & $ \leq 21.2 $ & $ \leq 21.7 $ & $ \leq 22.2 $ & $ \leq 22.7 $ \\\hline
      \end{tabular}}
      \vspace{0.3cm}
      \caption{}
      \label{tbl:remaining}
    \end{center}
\end{table}

In \cite{Boundaries}, \Cref{thm:known} is proven in two steps:
\begin{enumerate}
  \item A lower bound on the $\F_p$-Adams filtration of the classes in the kernel of $u_{8m-1}$ is established.
  \item This lower bound is compared to an upper bound on the $\F_p$-Adams filtration of elements in $\pi_{8m-1} \Ss_{(p)}$ which do not lie in the image of the $J$-homomorphism.
\end{enumerate}

More precisely, (1) and (2) are accomplished via the following proposition and definition from \cite{Boundaries}:

\begin{prop}[{\cite[Lemma 6.9 and Theorem 10.8]{Boundaries}}] \label{prop:w-filt}
  Suppose that $m \geq 3$. Then the kernel of the $p$-localized unit map
  \[u_{8m-1} : \pi_{8m-1} \Ss_{(p)} \to \pi_{8m-1} \MOfn_{(p)}\]
  is generated by image of $J$ and a single element $w \in \pi_{8m-1} \Ss_{(p)}$ which lies in $\F_p$-Adams filtration at least $2N_p-1$.\footnote{Note that $2N_p-1$ depends on $m$, though this is omitted from the notation. The general definition of $N_p$ is given in \cite[Definition 7.5]{Boundaries}.}
  Relevant values of $2N_p-1$ are summarized in \Cref{tbl:remaining}. 
\end{prop}

\begin{dfn} \label{dfn:h}
  Let $\Gamma_p(k)$ denote the minimal $m$ such that every $\alpha \in \pi_{k}\Ss_{(p)}$ with $\F_p$-Adams filtration strictly greater than $m$ is in the subgroup generated by the image of $J$ together with Adams's $\mu$-family (at the prime $2$).\footnote{However, note that the $\mu$-family elements only appear in stems congruent to $1$ and $2$ mod $8$, and in particular are absent from $\pi_{8m-1} \Ss_{(2)}$.}
\end{dfn}

In \Cref{tbl:remaining}, we have recorded known bounds on $\Gamma_2$ and $\Gamma_3$.
\begin{itemize}
  \item The sharp value of $\Gamma_2$ through $87$ and the bound in $95$ can be extracted from the extensive Adams spectral sequence computations of Isaksen, Wang and Xu \cite{IWX}.
  \item The bounds on $\Gamma_2$ above $95$ are obtained from \cite[Corollary 1.3]{DM3}.
  \item The sharp value of $\Gamma_3$ through $103$ can be extracted from Adams spectral sequence computations of Oka and Nakamura \cite{Oka, nakamura}. %and the computation of $3$-primary homotopy groups in Ravenel's green book \cite{Oka, nakamura, greenbook}.
  \item The bounds on $\Gamma_3$ above $103$ are obtained from \cite{cookware}.
\end{itemize}

Comparing the values of $2N_p - 1$ and $\Gamma_p (8m-1)$ in \Cref{tbl:remaining} and using \Cref{thm:known}, we see that $w \in \pi_{8m-1} \Ss_{(p)}$ must lie in the image of $J$ with the possible exception of the following six cases:
\[(p,8m-1) = (2,23),(3,23),(2,31),(3,39),(2,47),(2,127).\]
In order to handle these cases, as well as the $n=9$ case of \Cref{thm:unit-ker}, we make four essentially different arguments. We outline each of these arguments here.

\begin{enumerate}
  \item {\bf Dimensions $23$ (prime $2$), $31, 39$ and $47$.}
    In \Cref{sec:inch}, we make a mild improvement on the lower bound for the $\F_p$-Adams filtration of the classes in the kernel of $u_{8m-1}$ established in \cite{Boundaries}. This improvement resolves the cases in dimensions $31$, $39$ and $47$. It further implies that, modulo the image of $J$, the only element which can lie in the kernel of the $2$-localized map $u_{23}$ is $\eta^3 \kappabar$. By \Cref{thm:counterex} and \cite[Satz 1.7]{StolzBook}, it follows that $\eta^3 \kappabar$ does indeed lie in the kernel.
%\item Dimension $24$ at the prime $2$.
  %Using the argument in (1) we see that modulo the image of $J$, the only element of $\pi_{23} \Ss_{(2)}$ which might map to zero under the $2$-local unit map is $\eta^3 \kappabar$. We then show that, if $\eta^3 \kappabar$ did not map to zero in $\pi_{23} \MO\langle 12 \rangle_{(2)}$, there would exist a smooth closed $11$-connected $24$-manifold whose Witten genus violates the integrality conditions provided by the string orientation of $\tmf$ \cite{AHR}. This appears in \Cref{sec:witten}.
  %% with certain Pontryagin numbers. Any such manifold admits a string structure, hence a Witten genus, which is an integral modular form. However, the Ando-Hopkins-Rezk string orientation of $\tmf$ \cite{AHR} provides a restriction on the image of the Witten genus: not every modular form may be the Witten genus of a string manifold. We show in \Cref{sec:witten} that the Witten genus of the putative manifold above does not lie in the image of the Witten genus. This is the desired contradiction.
  \item {\bf Dimension $17$, i.e. $n = 9$.}
    In \Cref{sec:2-9}, we analyze the exceptional situation in dimension $17$.
    We first obtain an upper bound on the kernel of $u_{17}$ by considering the composition
    \[\pi_{17} \Ss \to \pi_{17} \MO\langle 9 \rangle \to \pi_{17} \MO \langle 8 \rangle \to \pi_{17} \tmf,\]
    noting that the kernel is generated by $\eta \eta_4$ and the image of $J$.
    In order to show $\eta\eta_4$ lies in the kernel of $u_{17}$, we examine the homotopy power operation $P^9$.
    Using the Steenrod squares on the $\mathrm{E}_2$-page of the $\F_2$-Adams spectral sequence, we find that $P^9$ sends $\eta \sigma$ to an element of $\pi_{17} \Ss$ which does not lie in the image of $J$. Since $\eta \sigma$ is in the image of $J$, we learn that it maps to zero in $\pi_{8} \MO \langle 9 \rangle$ and therefore $P^9(\eta \sigma)$ must also go to zero in $\pi_{17} \MO \langle 9 \rangle$, which concludes the argument.
  \item {\bf Dimension $23$ (prime $3$).}
    In \Cref{sec:3-3}, we study the map of $\F_3$-Adams spectral sequences induced by the composition
    \[\MO\langle12\rangle \to \MO\langle 8 \rangle \to \tmf\]
    of the canonical map with the Ando--Hopkins--Rezk string orientation of $\tmf$ \cite{AHR}.
    The string orientation provides the leverage necessary to conclude that the kernel of the $3$-localized map $u_{23}$ is equal to the image of $J$.
This appears in \Cref{sec:3-3}.
  \item {\bf Dimension $127$.}
    In \Cref{sec:medium} we make a short homological argument which shows there are no elements of $\F_2$-Adams filtration $49$ in the $127$-stem. It follows that $w \in \pi_{127} \Ss_{(2)}$ must be of $\F_2$-Adams filtration at least $50$. Consulting \Cref{tbl:remaining}, we find that it must lie in the image of $J$.
\end{enumerate}

\section{A homotopy argument} \label{sec:inch}
In order to finish the proof of \Cref{thm:unit-ker} in dimensions $31, 39$ and $47$, as well as the prime $2$ case of dimension $23$, we will need to improve the $\F_p$-Adams filtration bound from \cite[Theorem 10.8]{Boundaries}. As in the proof of that bound, this proof relies on the existence and properties of the category of synthetic spectra. We suggest the reader consult \cite{Pstragowski} for a general introduction to synthetic spectra and \cite[Section 9]{Boundaries} or \cite{cookware} for a more computational viewpoint.\footnote{In this section all synthetic spectra will be $\F_p$-synthetic spectra.} Since the proof we give is essentially a refinement of the argument in \cite[Section 10]{Boundaries} we will assume the reader is generally familiar with that argument.

In \cite[Lemma 6.9]{Boundaries} a Toda bracket expression for an element $w$ which generates the kernel of the unit map
\[\pi_{8m-1} \Ss_{(p)} \to \pi_{8m-1} \MOfn_{(p)} \]
(modulo the image of $J$) is given. We begin by recalling this expression. 

%% the Toda bracket expression for $w$ obtained in \cite[Lemma 6.9]{Boundaries}. First, we need a definition.First, we need a definition.

%% The improvemet in Adams filtration comes from producing a slightly better synthetic lift of the Toda bracket defining the element $w$ which generates the 

%% modulo the image of $J$.
%% We begin by recalling the Toda bracket expression for $w$ obtained in \cite[Lemma 6.9]{Boundaries}. First, we need a definition.

\begin{dfn} \label{dfn:skeleton}
Let
$$M \longrightarrow \Sigma^{\infty} \Oo \mathrm{\langle 4m-1 \rangle}$$
denote the inclusion of an $(8m-1)$-skeleton of $\Sigma^{\infty} \Oo \mathrm{\langle 4m-1 \rangle}$. By the inclusion of an $(8m-1)$-skeleton, we mean in particular that the induced map
$$ (\F_p)_* (M) \longrightarrow (\F_p)_* (\Sigma^{\infty} \Oo \mathrm{\langle 4m-1 \rangle})$$
is an isomorphism for $* \leq 8m-1$ and that $(\F_p)_* (M) \cong 0$ for $* > 8m-1$.
The generator $x \in \pi_{4m-1} \left(\sOf \right) \cong \Z$ is the image of some class in $\pi_{4m-1} M$, which by abuse of notation we also denote by $x$.  We additionally abuse notation by using $J$ to denote the composite map $$M \longrightarrow \Sigma^{\infty} \Oo \mathrm{\langle 4m-1 \rangle} \stackrel{J}{\longrightarrow} \mathbb{S},$$
where
\[ \Sigma^{\infty} \Oo \mathrm{\langle 4m-1 \rangle} \stackrel{J}{\longrightarrow} \mathbb{S}\]
is adjoint to the map
\[ \Oo \mathrm{\langle 4m-1 \rangle} \to \Oo \xrightarrow{J} \GL_1 (\Ss). \]
\end{dfn}

\begin{cnstr} \label{cnst:diagram}
  Consider the following diagram,
  \begin{center}
    \begin{tikzcd}[column sep = huge]
      \mathbb{S}^{8n-2} \arrow{r}{2}
      \arrow[rrdd, bend right=20,""{name=D}] \arrow[rrdd, bend right=20,swap,"0"] \arrow[rr,bend left = 30,""{name=U},"0"]
      \arrow[rrdd, bend right=20,""{name=D}] \arrow[rrdd, bend right=20,swap,"0"] \arrow[rr,bend left = 30,""{name=U},"0"]
      & \mathbb{S}^{8n-2} \arrow[Leftrightarrow, from=D, "h"] \arrow[Leftrightarrow, from=U, "f"] \arrow{r}{xJ(x)} \arrow[rdd,""{name=MU},"J(x)^2" description] \arrow[rdd,swap, ""{name=MD},"J(x)^2" description]
      & M \arrow{dd}{J}  \arrow[bend left=20,Leftrightarrow, from=MU, swap, "g"] \\ \\
      & & \mathbb{S}.
    \end{tikzcd}
    \label{diagram1}
  \end{center}
  where the homotopies $f,g$ and $h$ are chosen as follows:
  \begin{itemize}
  \item $f$ is an arbitrary nullhomotopy.
  \item $g$ is the canonical homotopy associated to the fact that $J$ is a map of $\Ss$-modules.
  \item $h$ is the canonical nullhomotopy given by the $\mathbb{E}_\infty$-ring structure on $\Ss$.
  \end{itemize}
  This provides a homotopy from $0$ to itself which defines the map $w: \Ss^{8n-1} \to \Ss^0$, well defined modulo the image of $J$. By \cite[Lemma 6.9]{Boundaries}, $w$ generates the kernel of the unit map
  \[u_{8m-1} : \pi_{8m-1} \Ss \to \pi_{8m-1} \MOfn. \]
\end{cnstr}

The Adams filtration bound in \cite[Theorem 10.8]{Boundaries} arises via the construction of a synthetic lift of the diagram in \Cref{cnst:diagram}. Our improvement in Adams filtration will come from producing a slightly better synthetic lift. 

\begin{rec}
  
  %% Recall the category of synthetic spectra, introduced in \cite{Pstragowski} and elaborated on in \cite[Section 9]{Boundaries}. We will not review the basic properties of synthetic spectra here and instead refer the reader to the summary given in \cite[Section 9]{Boundaries}. In this section, all synthetic spectra will be $\F_p$-synthetic spectra.
  As in \cite[Construction 10.5]{Boundaries}, once we move over to the category of synthetic spectra there is a lift of the map of synthetic spectra
  \[ \nu J : \nu M \to \Ss^{0,0} \]
  along the map $\tau^{N_p} : \Ss^{0,-N_p} \to \Ss^{0,0}$ to a map
  $ \nu M \to \Ss^{0,-N_p} $.
  which we view as a map
  \[ \wt{J} : \Sigma^{0,N_p} \nu M \to \Ss^{0,0}. \]
\end{rec}

The main result of this section is the following:

\begin{lem}
  Suppose that $\wt{J}(\nu(x)) = \tau^N z$. Then the Toda bracket $w$ has Adams filtration at least $2N_p + N - 1$.
\end{lem}

\begin{proof}
  Let us fix the following notation: we set $y = \wt{J}(\nu(x))$, $a = 8n-2$ and $b = 8n-2 + 2N_p + N$.
  Then, we construct the following diagram in synthetic spectra:
  \begin{center}
    \begin{tikzcd}[column sep = huge, row sep = huge]
      \mathbb{S}^{a,b} \arrow{r}{2}
      \ar[rd, ""{name=D}]
      \ar[rd, swap, "0"]
      \ar[rr, bend left = 30, ""{name=U},"0"]
      \ar[rr, bend left = 30, ""{name=U},"0"]
      & \mathbb{S}^{a,b}
      \ar[d, "z^2"]
      \arrow[dr,""{name=MU},"yz" description]
      \ar[Leftrightarrow, from=D, "\tilde{h}"]
      \ar[Leftrightarrow, from=U, "\tilde{f}"]
      \ar{r}{\nu(x) z}
      & \Sigma^{0,N_p}\nu(M)
      \arrow{d}{\wt{J}}
      \arrow[swap, Leftrightarrow, from=MU,"\tilde{g}"] \\
      & \Ss^{0,N}
      \ar[r, "\tau^N"]
      \arrow[shorten <= 0.7em, swap,Leftrightarrow, from=MU, "\tilde{k}"]
      & \Ss^{0,0} 
    \end{tikzcd}
  \end{center}
  where the homotopies $\tilde{f},\tilde{g}$ and $\tilde{h}$ are chosen as follows:
  \begin{itemize}
  \item $\tilde{f}$ is an arbitrary nullhomotopy, which exists as a consequence of the fact that $\pi_{8n-2,8n-2+\gamma} (\nu M) = 0$ for all $\gamma \geq 2$ \cite[Proof of Proposition 10.7]{Boundaries}.
  \item $\tilde{g}$ is the canonical homotopy that expresses the fact that $\tilde{J}$ is a map of right $\Ss^{0,0}$-modules.
  \item $\tilde{h}$ is the canonical nullhomotopy that comes from the fact that $\Ss^{0,0}$ is an $\mathbb{E}_\infty$-ring in the symmetric moniodal $\infty$-category $\mathrm{Syn}_{\HFp}$.
  \item $\tilde{k}$ is the composite of a homotopy expressing that $z \tau^N = y$ and the natural homotopy expressing that composition with $z$ is homotopic to multiplication by $z$.
  \end{itemize}

  This diagram determines a homotopy of $0$ with itself, hence a map
  \[\wt{w}: \Ss^{a+1,b} \to \Ss^{0,0}.\]
  
  On applying $\tau^{-1}$, the above diagram recovers \Cref{cnst:diagram}, so $\wt{w}$ maps to $w$ under $\tau^{-1}$. The desired Adams filtration bound now follows from \cite[Corollary 9.21]{Boundaries}.
\end{proof}

Using this lemma, we are able to finish the proof of \Cref{thm:unit-ker} in the promised cases.

\begin{prop}
  When $(p,8m-1) = (2,31), (3,39)$ and $(2,47)$ the element $w$ lies in the image of $J$.
  On the other hand, when $(p,8m-1) =  (2,23)$, the image of the element $w$ in the cokernel of $J$ is $\eta^3\bar{\kappa}$.
\end{prop}

%%   \begin{center}
%%     \begin{tabular}{|c|c|c|}
%%       \hline $p$ & $n$ & Adams filtration bound \\\hline\hline
%%       2 & 3 & 7 \\\hline
%%       2 & 4 & 6 \\\hline
%%       2 & 6 & 15 \\\hline
%%     \end{tabular}
%%   \end{center}

%%   In particular, we learn that for $p=2, n=3$ we have $w \in \{0, \eta^3 \bar{\kappa}\}$ and for $p=2, n=4,6$ that $w = 0$.
%% \end{prop}

\begin{proof}
  Throughout this proof, we will freely make use of \cite[Theorem 9.19]{Boundaries} in order to translate between knowledge of the $\F_p$-Adams spectral sequence and knowledge of $\F_p$-synthetic homotopy groups.
  
  {\bf Dimension 23, prime 2.} In this case, $J(x) = \zeta_{11}$ and we can determine that $\wt{J}(\nu(x)) = \tau^2 \wt{\zeta_{11}}$ because there is no $\tau$-torsion in this bidegree. Thus, $w$ has $\F_2$-Adams filtration at least 7. It follows from \Cref{thm:counterex} and \cite[Satz 1.7]{StolzBook} that the image of $w$ in the $2$-localized cokernel of $J$ must be nonzero. Using our restriction on its $\F_2$-Adams filtration, we conclude that $w$ must be equal to $\eta^3 \bar{\kappa}$ in the $2$-localized cokernel of $J$. 

  {\bf Dimension 31, prime 2.} In this case, $J(x) = \rho_{15}$ and we can determine that $\wt{J}(\nu(x)) = \tau \wt{\rho_{15}}$ because there is no $\tau$-torsion in this bidegree. Thus, $w$ has Adams filtration at least $6$ and so it must be in the image of $J$ (see \Cref{tbl:remaining}). 
  
  {\bf Dimension 39, prime 3.} In this case, $J(x) = \alpha_5$ and we can determine that $\wt{J}(\nu(x)) = \tau^2 \wt{\alpha_5}$ because there is no $\tau$-torsion in this bidegree. Thus, $w$ has Adams filtration at least $7$ and so it must be in the image of $J$ (see \Cref{tbl:remaining}). 

  {\bf Dimension 47, prime 2.} Once again, $\wt{J}(\nu(x))$ lands in a bidegree with no $\tau$-torsion where every element is divisible by $\tau^2$. Thus, $w$ has Adams filtration at least $15$ and so it must be in the image of $J$ (see \Cref{tbl:remaining}).
\end{proof}

\section{The case of dimension $17$} \label{sec:2-9}
The goal of this section is to prove the following theorem, which shows that the dimension $17$ is exceptional:

\begin{thm} \label{thm:2-9}
  The kernel of the unit map
  \[u_{17} : \pi_{17} \Ss \to \pi_{17} \MO\langle 9 \rangle\]
  is generated by the image of $J$ and $\eta \eta_4 \in \pi_{17} \Ss$.
\end{thm}

%The proof of this theorem relies on the following sequence of maps of $\mathbb{E}_\infty$-rings
%\[ \Ss \to \MO\langle 9 \rangle \to \MO\langle 8 \rangle \to \tmf, \]
%which are respectively, the unit, the canonical map and the Ando--Hopkins--Rezk string orientation.
We begin the proof of \Cref{thm:2-9} with the following bounds on the size of the kernel of $u_{17}$:

\begin{lem}
  The kernel of the unit map
  \[\pi_{17} \Ss \to \pi_{17} \MO \langle 9 \rangle \]
  is either equal to the image of $J$ or the subgroup of $\pi_{17} \Ss$ generated by the image of $J$ and $\eta \eta_4$.
\end{lem}

\begin{proof}
  It is well known that in degrees $8$ and above the image of $J$ is in the kernel of the unit map for $\MO \langle 9 \rangle $.
  % follows from \cite[Lemma 6.1(ii)]{GRAbelianQuotients} that the kernel contains the image of $J$.
  %One may read off from the computations of \cite{Bauertmf} that the kernel of $\pi_{17}$ applied to the sequence of maps above is generated by the image of $J$ and $\eta \eta_4$.
  Composing the Ando-Hopkins-Rezk string orientation \cite{AHR} with the canonical map $\MO \langle 9 \rangle \to \MO \langle 8 \rangle$, we obtain an $\mathbb{E}_\infty$-ring map $\MO \langle 9 \rangle \to \tmf$.
  One may read off from the computations of \cite{Bauertmf} that the kernel of
  \[\pi_{17} \Ss \to \pi_{17} \MO \langle 9 \rangle \to \pi_{17} \tmf\]
  is generated by the image of $J$ and $\eta \eta_4$, from which the proposition follows.
\end{proof}

It now suffices to show that the kernel of the unit map
\[\pi_{*} \Ss \to \pi_{*} \MO \langle 9 \rangle\]
contains an element not in the image of $J$ in dimension 17.
To prove this, we will use the fact that this kernel is closed under spherical power operations.
The power operation of interest to us is described in the following proposition.

\begin{prop}[{\cite[Table V.1.3]{BMMS}}] \label{prop:power-op}
  Let $R$ be an $\mathbb{E}_\infty$-ring. There is a natural, not necessarily additive, operation $P^{9}$ from $\pi_{8} R$ to $\pi_{17} R$, with indeterminacy. The indeterminacy of $P^9 (x)$ is $\eta x^2$.

  Moreover, if $x \in \pi_8 R$ is detected by $a \in \mathrm{E}_2 ^{s,8+s}$ on the $\mathrm{E}_2$-page of the $\HFt$-Adams spectral sequence, then $P^9 x$ lies in $\HFt$-Adams filtration at least $2s-1$ and its image in $\mathrm{E}_2 ^{2s-1,17+2s-1}$ is $\Sq^{9} a$.
\end{prop}
%
%The following lemma will allow us control the indeterminacy.
%
%\begin{lem} \label{lem:2-eta-0}
%  Let $R$ be an $\mathbb{E}_\infty$-ring. Then if $x \in \pi_{2n} R$ satisfies $2x = 0$, we have $\eta x^2 = 0$.
%\end{lem}
%
%\begin{proof}
%  This follows from \cite[Proposition V.1.12]{BMMS}.
%\end{proof}

\begin{proof}[Proof of \Cref{thm:2-9}]
  Applying this power operation to $\eta \sigma \in \pi_8 \Ss$, which lies in the image of $J$ and thus the kernel of
  \[ \pi_{8} \Ss \to \pi_{8} \MO \langle 9 \rangle, \]
  %by \cite[Lemma 6.1(ii)]{GRAbelianQuotients},
  we learn that $P^9 (\eta \sigma)$ is in the kernel of the unit map for $\MO\langle 9 \rangle$. Since $\eta (\eta \sigma)^2 = 0$, the operation $P^9 (\eta \sigma) \in \pi_{17} \Ss$ has no indeterminacy.

  The class $P^9(\eta \sigma)$ is detected on the $\mathrm{E}_2$-page of the $\F_2$-Adams spectral sequence by $Sq^9 (h_1h_3)$. Using the Cartan formula and the fact that $\Sq^{2^i} (h_i) = h_{i+1}$, cf. \cite[Proposition 1.4(i)]{BMMS}, it follows that 
  \[\Sq^9 (h_1 h_3) = \Sq^1 (h_1) \Sq^8 (h_3) + \Sq^2 (h_1) \Sq^7 (h_3) = h_1 ^2 h_4 + h_2 h_3^2 = h_1^2 h_4. \]
  Since no element detected by $h_1^2h_4$ lies in the image of $J$, we are done.
\end{proof}

\section{The case of dimension $23$ at the prime $3$} \label{sec:3-3}
In this section, we will show that, although the kernel of the unit map $\pi_{23}\Ss \to \pi_{23}\mathrm{MO}\langle 12 \rangle$ contains an exceptional element at the prime 2, at the prime 3 the kernel contains only the image of $J$.
This is the final step in the proof of the dimension $23$ case of \Cref{thm:unit-ker}.

We will prove this by directly computing of the $\F_3$-Adams spectral sequence for $\mathrm{MO}\langle 12 \rangle$. As at the prime $2$, one of the key techniques in this argument is comparison with $\mathrm{tmf}$ via the Ando--Hopkins--Rezk string orientation \cite{AHR}. The first step we take is to compute the homology of $\MOtw$ as an $\mathcal{A}_*$-comodule in a range.

As is common in odd primary Adams spectral sequence computations, everything will be implicitly $3$-completed and, we will make use of the $\dot{=}$ notation, which means that an equation holds up to a multiplication by a $3$-adic unit. Similarly, since we do not keep track of constants all claims in this section should be regarded as true up to multiplication by a $3$-adic unit.

\begin{lem}
  \label{lemm:mo12-homology}
  In degrees $\leq 25$, the $\F_3$-homology of $\mathrm{MO}\langle 12 \rangle$ has the following properties:
  \begin{enumerate}
  \item It's isomorphic to $\F_3 \oplus (\F_3)_*(\Sigma^{12} ko) \oplus \Sigma^{24}\F_3 $ as an $\mathcal{A}_*$-comodule.
  \item The only nontrivial product is the square of the generator in degree $12$, which is equal to a generator of the third summand.
  \item On $\F_3$-homology, the composition of the canonical map with the string orientation
    \[\MO\langle 12 \rangle \to \MO \langle 8 \rangle \to \tmf\]
    is only nonzero on the unit.
  \end{enumerate}
\end{lem}

\begin{proof}
  %% reduce to BO
  We begin by showing that, in degrees $\leq 25$, the Thom isomorphism
  \[ \F_3 \otimes \MOtw \simeq \F_3 \otimes \Sigma_+^\infty \mathrm{BO}\langle 12 \rangle, \]
  which is an equivalence of $\mathbb{E}_\infty$-rings, preserves the $\mathcal{A}_*$-comodule structure.
  %The ring structure on $\MOtw$ is the one coming from the infinite-loop structure on $\mathrm{BO}\langle 12 \rangle$ therefore it's compatible with the Thom isomorphism.
  To do this, we just need to show that the action of the Steenrod algebra on the Thom class $u \in (\F_3)^0 (\MOtw)$ is trivial through degree $25$.
  Since the Steenrod algebra is generated by $\beta, P^1$ and $P^3$ in this range, it suffices to show that the action of these operations on $u$ is zero.
  
  To show this, we note that $u$ is the pullback of another class $u \in (\F_3)^0 (\tmf)$ along the composition 
  $ \MOtw \to \MO\langle 8 \rangle \to \tmf $.
  From \cite[Section 4.1]{Culvertmf}, we can extract that through dimension $12$ the cohomology of $\tmf$ is given by $\F_3\{u\} \oplus \F_3\{b_4\} \oplus \F_3\{z\}$ where $|u| = 0$, $|b_4| = 8$ and $|z| = 12$ with Steenrod action $z \dot{=} P^1(b_4) \dot{=} P^3(u)$. For degree reasons $\beta(u) = P^1(u)= 0$ and $b_4$ maps to zero in $(\F_3)^8(\MOtw) = 0$, therefore the Steenrod operations $\beta, P^1$ and $P^3$ act trivially on $u$ in $\MOtw$.

  %the Pontryagin classes of the universal bundle over $\mathrm{BO}\langle 12 \rangle$ are divisible by 3 through this range.
  %Since we are working at the prime $3$ it will suffice to show that the Chern classes of the complexification of the universal bundle are divisible by 3. This follows from the main result of \cite{Singer}.
  
  %% use Goodwillie tower
  Using the Goodwillie tower of the identity for $\mathbb{E}_\infty$-rings, as worked out in \cite{KuhnTAQ}, we obtain a tower of nonunital $\mathbb{E}_\infty$-rings:
  \begin{center}
  \begin{tikzcd}[column sep = small]
    & & D_3(\tau_{\ge 12} ko) \ar[d] &
    D_2(\tau_{\ge 12} ko) \ar[d] &
    \tau_{\ge 12} ko \ar[d, "\simeq"] \\
    \Sigma^\infty \mathrm{BO}\langle 12 \rangle \ar[r] &
    \cdots \ar[r] &
    Q_3 \ar[r] &
    Q_2 \ar[r] &
    Q_1 .
  \end{tikzcd}
\end{center}
  For connectivity reasons, through degree $25$ we only need to work with $Q_2$.  
  %The statement about vanishing of products follows from the fact that $Q_1$ is a non-unital ring with trivial product (i.e. a square-zero extension of $0$).
  Since $Q_1$ is the stabilization of $\Sigma^\infty \BO \langle 12 \rangle$, the product on $Q_1$ is zero.

  Note that
\[ (\F_3)_{*}(D_2(\tau_{\geq 12} ko)) \cong \begin{cases} 0 & * \leq 23 \text{ or } * = 25 \\ \F_3 & * = 24 \end{cases} \]
  and let $x$ denote a generator of $(\F_3)_{12}(\mathrm{BO}\langle 12 \rangle)$. In order to finish the proof of (1) we only need to show that the vertical map into $Q_2$ is injective on $\F_3$-homology in degree 24. In fact, this would follow from knowing that $x^2$ is nonzero, which itself would imply (2), given that we know the product on $Q_1$ is zero.

  The bottom Postnikov truncation $\mathrm{BO}\langle 12 \rangle \to K(\Z_{3}, 12)$ is compatible with the product structure, so it now suffices to show that $\iota^2$ is nonzero where $\iota$ is the canonical class in $(\F_3)_{12} (K(\Z_{3}, 12))$. In order to show this class is nonzero, we consider its coproduct
  $$ \triangle(\iota^2) = \iota^2 \otimes 1 + 2(\iota \otimes \iota) + 1 \otimes \iota^2, $$
  where the middle term is clearly nonzero.

  Now we turn to (3).
  %We've already shown that $(\F_3)_*(\MOtw)$ splits into three pieces through the range we're considering.
  Applying $\F_3$-cohomology to the map $\MO \langle 12 \rangle \to \tmf$, we obtain a map of $\mathcal{A}^*$-modules
  \[(\F_3)^* \tmf \to (\F_3)^* (\MOtw). \]%\to \F_3 \oplus (\F_3)^* (\Sigma^{12} ko) \oplus \Sigma^{24} \F_3.\]
%
  %At the prime $2$, the cohomology of $\tmf$ is cyclic as an $\mathcal{A}^*$-module, so it would be sufficient to determine it on the unit. At the prime 3, the situation is a bit more complicated but still manageable.

  Since $(\F_3)^*(\tmf)$ has a two stage filtration by cyclic $\mathcal{A}^*$-modules with generators in degrees $0$ and $8$, respectively \cite[Theorem 21.5(2)]{Rezktmf}, and $(\F_3)^*(\Sigma^{12}ko) \oplus \Sigma^{24}\F_3$ begins in degree 12, it follows that through degree $25$ the map
  \[(\F_3)^* \tmf \to (\F_3)^* (\MOtw)\]
  factors through the unit.
\end{proof}

As a consequence of this lemma, we learn that the $\mathrm{E}_2$-page of the $\F_3$-Adams spectral sequence for $\mathrm{MO}\langle 12 \rangle$ takes the form shown in \Cref{fig:mo12e2}. Next, we determine the easy differentials in this spectral sequence.

\begin{sseqdata}[ name = moASS, xscale=0.45, yscale=0.9, x range = {0}{24}, y range = {0}{7}, x tick step = 4, y tick step = 2, axes type = frame, class labels = {left}, classes = fill, grid = go, x grid step = 4, y grid step = 2, Adams grading]

  \class(0,0)
  \class["a_0" right](0,1) \structline
  \foreach \i in {2,...,8} { \class(0, \i) \structline }

  \class["h_0" right](3,1) \structline(0,0)
  \class["\alpha_2" above](7,2)
  \class["b_0" below](10,2) \class(10,3) \structline \structline(7,2)

  \class["h_1" below](11,1)
  \class(11,2) \structline \class(11,3) \structline

  \class(13,3) \structline(10,2)
  \class["\alpha_4" above](15,4)
  \class(17,4) \class(20,5) \structline \class(20,4) \structline \class(23,5) \structline
  \class(18,2) \class(21,3) \structline
  \class["\alpha_5"](19,5) \class(22,6) \structline \class(22,5) \structline

  \class["u"](23,3)
  \foreach \i in {1,...,3} { \class(23, 3+\i) \structline }

  \class[red, "x_{12}" above right](12,0)
  \foreach \i in {1,...,8} { \class[red](12, \i) \structline }
  \class[red, "x_{16}" right](16,1)
  \foreach \i in {1,...,8} { \class[red](16, 1+\i) \structline }
  \class[red, "x_{20}" below right](20,2)
  \foreach \i in {1,...,8} { \class[red](20, 2+\i) \structline }
  \class[red, "x_{24}" below left](24,3)
  \foreach \i in {1,...,8} { \class[red](24, 3+\i) \structline }

  \class[blue, "x_{12}^2" above left](24,0)
  \foreach \i in {1,...,8} { \class[blue](24, \i) \structline }

  \d[cyan]2(11,1)
  \d[cyan]2(18,2)
  \d[cyan]2(21,3)(20,5)
  \d[cyan]2(23,3)(22,5)
  \d[cyan]2(23,4)(22,6)
  \d[cyan]2(12,0)
  \d[cyan]2(12,1)      
  \d[red]3(16,1)
  \d[red]3(20,2)
  \d[cyan]2(24,3)(23,5,2)
  \d[cyan]2(24,4)
\end{sseqdata}

\begin{sseqdata}[ name = tmfASS, xscale=0.45, yscale=0.9, x range = {0}{24}, y range = {0}{7}, x tick step = 4, y tick step = 2, axes type = frame, class labels = {left}, classes = fill, grid = go, x grid step = 4, y grid step = 2, Adams grading]

  %% stem 0-4
  \class(0,0)
  \class["a_0" right](0,1) \structline
  \foreach \i in {2,...,10} { \class(0, \i) \structline }
  \class["h_0" right](3,1) \structline(0,0)

  %% stem 5-8
  \class["\alpha_2" above](7,2)
  \class(8,0)
  \foreach \i in {1,...,10} { \class(8, \i) \structline }

  \d[cyan]2(8,0)
  
  %% stem 9-12
  \class(10,2)
  \class(10,3) \structline \structline(7,2)
  \class(11,1) \structline(8,0)
  \class["c_6"](12,3)
  \foreach \i in {1,...,10} { \class(12, 3+\i) \structline }

  \d[cyan]2(11,1)
  
  %% stem 13-16
  \class(13,3) \structline(10,2)
  \class["\alpha_4" above](15,4) \structline(12,3)
  \class(16,2)
  \foreach \i in {1,...,10} { \class(16, 2+\i) \structline }

  \d[cyan]2(16,2)

  %% stem 17-20  
  \class(17,4)
  \class(18,2)
  \class["\alpha_5" above](19,5)
  \class(20,4) 
  \class(20,5) \structline \structline(17,4)
  \class(20,3)
  \foreach \i in {1,...,10} { \class(20, 3+\i) \structline }
  
  \d[cyan]2(18,2)
  \d[cyan]2(20,3)
  
  %% stem 21-24
  \class(21,3) \structline(18,2)
  \class(22,5) 
  \class(22,6) \structline \structline(19,5)
  \class["u"](23,3)
  \class(23,4) \structline
  \class(23,5) \structline(20,4)
  \class["\Delta"](24,1)
  \foreach \i in {1,...,10} { \class(24, 1+\i) \structline }
  \class["c_6^2" above left](24,6)
  \foreach \i in {1,...,10} { \class(24, 6+\i) \structline }

  \d[cyan]2(21,3)(20,5)
  \d[cyan]2(23,3)(22,5)
  \d[cyan]2(23,4)(22,6)
  \d[green!70!black]4(24,1)
  
  %% stem 25+
  \class(25,6) \structline(22,5)
  \class(26,4) \structline(23,3)
  \class(27,2) \structline(24,1)
  \class(27,7) \structline(24,6)
\end{sseqdata}

\begin{sseqdata}[ name = synS, xscale=0.45, yscale=0.9, x range = {0}{24}, y range = {0}{7}, x tick step = 4, y tick step = 2, axes type = frame, class labels = {left}, classes = fill, grid = go, x grid step = 4, y grid step = 2, Adams grading]
  
  \class(0,0)
  \foreach \i in {1,...,8} { \class(0, \i) \structline }

  \class["\alpha_1" above](3,1) \structline(0,0)
  \class["\alpha_2" above](7,2)
  \class["\beta_1" below](10,2) \class[red](10,3) \structline \structline(7,2)

  % \class["h_1"](11,1)
  \class["\alpha_{3/2}" below right](11,2) %\structline
  \class(11,3) \structline

  \class(13,3) \structline(10,2)
  \class["\alpha_4" above](15,4)
  \class[red](17,4)
  \class[red](20,5) \structline
  \class(20,4) \structline
  \class(23,5) \structline
  %\class(18,2)
  %\class(21,3) \structline
  \class["\alpha_5" above](19,5)
  \class[red](22,6) \structline
  \class[red](22,5) \structline

  % \class["u"](23,3)
  \class(23,5)
  \class(23,6) \structline

  \class(25,6) \structline(22,5)

  %\d[cyan]2(11,1)
  %\d[cyan]2(18,2)
  %\d[cyan]2(21,3)(20,5)
  %\d[cyan]2(23,3)(22,5)
  %\d[cyan]2(23,4)(22,6)
\end{sseqdata}

\begin{figure}
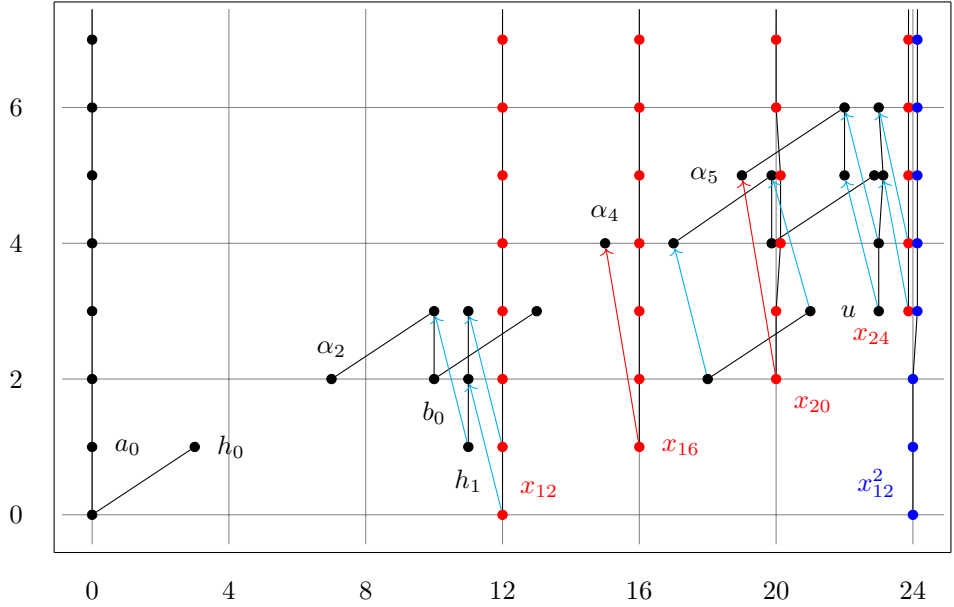

  \centering
  The $\F_3$-Adams spectral sequence for $\MOtw$ \\\vspace{6pt}
  \printpage[ name = moASS, page = 0 ]
  \caption{The $\F_3$-Adams spectral sequence for $\MOtw$. Classes in black come from the unit copy of $\F_3$. Classes in red come from the copy of $(\F_3)_*(ko)$ in degree $12$. Classes in blue come from the copy of $\F_3$ in degree 24. The nontrivial product in homology from \Cref{lemm:mo12-homology} lets us conclude that $x_{12}^2$ is as labeled.}
  \label{fig:mo12e2}
  \label{fig:mo12}
\end{figure}

\begin{lem} \label{lem:easy-diff}
  Through degree $24$, the differentials in the $\F_3$-Adams spectral sequence for $\mathrm{MO}\langle 12 \rangle$ fit into one of the following two families:
  \begin{enumerate}
  \item differentials induced from the sphere, all of which occur,
  \item the following extra differentials:
    \begin{alignat*}{3}
      d_2(x_{12}) & \dot{=} a_0h_1 &\qquad\qquad d_3(x_{16}) & \dot{=} \alpha_4 &\qquad\qquad d_3(x_{20}) & \dot{=} \alpha_5 \\
      d_2(x_{24}) & \dot{=} a_0^2u & d_5(x_{12}^2) & \dot{=} ?
    \end{alignat*}
  \end{enumerate}
  which are displayed in \Cref{fig:mo12}.
\end{lem}

After proving this lemma, the final task of this section will be to show that $x_{12}^2$ is a permanent cycle.

\begin{proof}
  First, we note that for degree reasons nothing can interfere with the differentials induced from the sphere. We also know that the image of $J$ elements $\alpha_{3/2}, \alpha_4$ and $\alpha_5$ must each map to zero in $\MOtw$ as well. For each of these there is a unique possible differential which could enforce that relation. It remains to show that $d_2(x_{24}) \dot{=} a_0^2u$.

  In the sphere we can use Moss's theorem \cite{Moss} to conclude that $\langle \alpha_5, \alpha_1 , 3 \rangle$ is detected by $a_0^2u$ in the $\F_3$-Adams spectral sequence. Since the indeterminacy is just 3 times this same element we find that this bracket is equal to $\alpha_{6/2}$ up to a unit. The class $\alpha_5$ maps to zero in $\MOtw$ therefore $\alpha_{6/2}$ becomes divisible by 3. Examining \Cref{fig:mo12}, this can only happen if $\alpha_{6/2}$ is zero. Thus, we know it gets hit by some differential.
  %\todo{Perhaps better: we know that a generator for the image of $J$ is detected by $a_0 ^2 u$, since there is nothing else to detect it, so we immediately learn that $a_0 ^2 u$ dies. This can be a one-line argument. Downside is that it doesn't set up the rest of the proof as well.}
  We will finish the proof by using synthetic spectra to bound the length of the Adams differential that hits $a_0^2u$.

The first step is to lift the Toda bracket above to one in the synthetic category.
Using \cite[Theorem 9.19]{Boundaries}, we may compute the $\F_3$-synthetic homotopy of the sphere through $24$ using the known computation of its $\F_3$-Adams spectral sequence in this range.
The result is displayed in \Cref{fig:synS}.
We next fix some names for specific elements of $\pi_{*,*}\Ss$. 
\begin{itemize}
\item Let $\wt{3} \in \pi_{0,1}\Ss$ denote the unique element in that degree which
  maps to $3$ in $\pi_0\Ss$ and $a_0$ in $\pi_{0,1}(C\tau) \cong \Ext_{\mathcal{A}}^{1,4}(\F_3,\F_3)$.
\item Let $\alpha_1 \in \pi_{3,4}\Ss$ denote the unique element in that degree which
  maps to $\alpha_1$ in $\pi_3\Ss$ and $h_0$ in $\pi_{3,4}(C\tau) \cong \Ext_{\mathcal{A}}^{1,4}(\F_3,\F_3)$.
\item Let $\alpha_5 \in \pi_{19,24}\Ss$ denote the unique element in that degree which maps to $\alpha_5$ in $\pi_3\Ss$ and a generator of $\pi_{19,24}(C\tau)$.
\end{itemize}

Since the image of $\alpha_1 \alpha_5$ in $\pi_{22,28}(C\tau)$ is hit by a $d_2$ differential and there are no classes above it we learn that $\tau \alpha_1 \alpha_5 = 0$.
Since $\pi_{3,5}\Ss = 0$ we learn that $\wt{3} \alpha_1 = 0$.
This means we can form the Toda bracket $x := \langle \tau \alpha_5, \alpha_1 , \wt{3} \rangle$ in synthetic spectra.
Upon inverting $\tau$, the bracket $x$ goes to the bracket $\langle \alpha_5, \alpha_1, 3 \rangle$, which is detected by $a_0 ^2 u$ on the $\mathrm{E}_2$-page.
We may therefore conclude that $x$ maps to $a_0^2u$ in the homotopy of $C\tau$.
If we show that the image of $x$ in $\nu \MOtw$ is simple $\tau$-torsion, then this will imply that $a_0^2u$ is hit by a $d_2$ differential in the $\F_3$-Adams spectral sequence for $\MOtw$.

Toda brackets are preserved by $\mathbb{E}_1$-ring maps, and $\Ss \to \nu \MOtw$ is a map of $\mathbb{E}_\infty$-rings, therefore we may make the following manipulations of Toda brackets (now considered in $\nu\MOtw$):
\[ \tau x \in \tau \langle \tau \alpha_5, \alpha, \wt{3} \rangle \subset \langle \tau^2 \alpha_5, \alpha, \wt{3} \rangle = \langle 0, \alpha, \wt{3} \rangle = \wt{3} \pi_{23, 26}(\MOtw), \]
  where the fact that $\tau^2\alpha_5 = 0$ follows from it getting hit by the $d_3$-differential off of $x_{20}$. Finally, since there is no possible nonzero $\wt{3}$-division for $\tau x$, we conclude that it is zero.

%% This Toda bracket lifts to the $\F_3$-synthetic Toda bracket $\wt{\alpha_{6/2}} \in \langle \tau \wt{\alpha_5}, \wt{\alpha}, \wt{3} \rangle$. \todo{Why? Explain tilde notation.} Then, since we know that $\wt{\alpha_5}$ is $\tau^2$-torsion in the bigraded homotopy groups of $\nu \MOtw$ by \cite[Theorem 9.19]{Boundaries}, we may conclude that $\wt{\alpha_{6/2}}\tau$ is divisible by $\wt{3}$. \todo{Indeterminacy?} The only way this can happen is if $\wt{\alpha_{6/2}}\tau = 0$. From this we can read off that $a_0^2u$ is hit by a $d_2$ in the Adams spectral sequence for $\MOtw$. The unique possible source for this differential is $x_{24}$.
  %% Next we use the fact that 
  %% First we calculate the $E_2$-term for $\mathrm{MO}\langle 12 \rangle$. This admits a three step filtration by $\Ss^{22}$, $\Sigma^{11}ko$ and $\Ss^0$. The attaching maps are zero on homology therefore the AAHSS has no differentials and we may simply superimpose these pictures
  %% This provides a map of skeleta $\mathrm{cof}(J) \to \mathrm{MO}\langle 12 \rangle $ which is injective on homology. We know what the attaching map does on homotopy so we can easily compute the Adams spectral sequence for $\mathrm{cof}(J)$. In particular, we learn that
\end{proof}

%% \begin{figure}
%%   \centering
%%   Adams spectral sequences for $\MOtw$ \\\vspace{6pt}
%%   \printpage[ name = moASS, page = 0 ]
%%   \caption{The $\F_3$-Adams spectral sequence for $\MOtw$. Classes are colored as in \Cref{fig:mo12e2}.}
%%   \label{fig:mo12}
%% \end{figure}

\begin{figure}
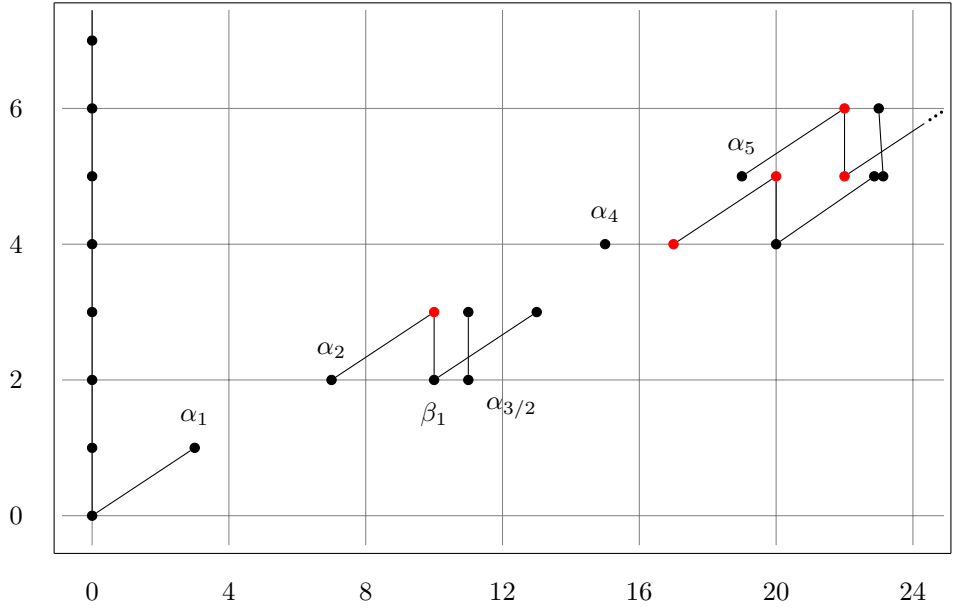

  \centering
  The bigraded homotopy groups of the $\F_3$-synthetic sphere. \\\vspace{6pt}
  \printpage[ name = synS, page = 0 ]
  \caption{The bigraded homotopy groups of the $\F_3$-synthetic sphere displayed in the $(t-s,s)$-plane.
    Black dots denote $\tau$-torsion-free classes,
    red dots denote $\tau^1$-torsion classes.
    Note that in order to reconstruct the group in a given bidegree one must examine all degrees lying above it.}
  \label{fig:synS}
\end{figure}

In dimension 12 we'll need slightly finer information than that provided by \cref{lem:easy-diff}.

\begin{lem}
  The element $c_6$ which generates $\pi_{12}\mathrm{tmf}_{3}$ lifts to $\pi_{12}\mathrm{MO}\langle 12 \rangle$, where it is detected (up to a unit) by $a_0^2x_{12}$ in the $\F_3$-Adams spectral sequence.
\end{lem}

\begin{proof}
  We begin by considering the following sequence of maps:
  $$ \Ss^{11} \xrightarrow{9} \Ss^{11} \xrightarrow{\alpha_{3/2}} \Ss^0 \xrightarrow{\iota} \tmf. $$
  Since each pairwise composite is nullhomotopic, we can form the Toda bracket $\langle 9, \alpha_{3/2}, \iota \rangle$. This Toda bracket has indeterminacy $9\pi_{12}(\tmf) + \pi_{12}(\Ss)\iota = (9\Z) \cdot c_6$. We will begin by showing that, up to a unit, $c_6$ is contained in this bracket.

  This bracket can be evaluated using the corresponding Massey product $\langle a_0^2, a_0h_1, \iota \rangle$. In particular, it will suffice to show that this Massey product is equal to $c_6$ as an element of the $\mathrm{E}_2$-page for $\tmf$. After consulting this $\mathrm{E}_2$-term we note that $c_6$ is the only element in its bidegree, so it will suffice to simply show that the bracket is nontrivial. In order to do this, we shuffle the bracket with $h_0$:
  $$ h_0 \langle a_0^2, a_0h_1, \iota \rangle = \langle h_0, a_0^2, a_0h_1 \rangle \iota \dot{=} \alpha_4 \iota. $$
  Finally, we note that the image of $\alpha_4$ in the $\mathrm{E}_2$-page for $\tmf$ is nontrivial \cite{Culvertmf}.
  From this we can read off that $\langle 9, \alpha_{3/2}, \iota \rangle = u c_6 $ for some $u \in \ZZ_{3}^\times$. The claim that $c_6$ lifts to $\MOtw$ now follows from the fact that the bracket $ \langle 9, \alpha_{3/2}, \iota' \rangle $ is defined where $\iota'$ is the unit of $\MOtw$. % This provides a lift of some element of the form $c_6 + 9?$ to $\mathrm{MO}\langle 12 \rangle$. Now, we know that some $\Z_{3}$-adic unit times $c_6$ lifts which implies that $c_6$ itself lifts.
  Since $c_6$ is a generator of $\pi_{12} \tmf_{3}$, the only possibility for an $\F_3$-Adams representative of its lift to $\pi_{12} \MOtw$ is (up to a unit) $a_0 ^2 x_{12}$.
%
  %Under the map of Adams spectral sequneces from $\mathrm{MO}\langle 12 \rangle$ to $\tmf$ induced by the string orientation, the class $x_{12}$ on the $\mathrm{E}_2$-page maps to zero by \Cref{lemm:mo12-homology}.
  %This implies that $c_6$ is detected in Adams filtraton at most $2$ in $\mathrm{MO}\langle 12 \rangle$. \todo{Why? And isn't it obvious that you're detected by $a_0 ^2 x_{12}$ up to a unit, since everything else is divisible by $3$, which $c_6$ isn't?}
  %The class $v_0^2x_{12}$ is the unique permanent on the 2-line or below therefore we get the desired result.
  %% This leads to the desired conclusion for $c_4^2$ and $c_6$. The hardest case is $c_4^3$. In tmf we have the following relations,
  %% $$ c_4^3 = \langle 3, \alpha_2, \tau \iota \rangle c_4^2 = \langle 3, \alpha_2, \tau \iota c_4^2 \rangle. $$
  %% We would like to incarnate the this last bracket in $\mathrm{MO}\langle 12 \rangle$. If we can do this then it would provide a lift of $c_4^3 + 3?$. The only thing we need to check is that $\alpha_2 c_4^2 = 0$. We can prove this by the following manipulations:
  %% \begin{align*}
  %%   \alpha_2c_4^2
  %%   &= \alpha_2 \langle 3, \alpha_4, \tau^2 \iota \rangle \\
  %%   &= \langle \alpha_2 , 3, \alpha_4 \rangle \tau^2 \iota \\
  %%   &= \alpha_? \tau^2 \iota = 0
  %% \end{align*}
\end{proof}

\begin{figure}
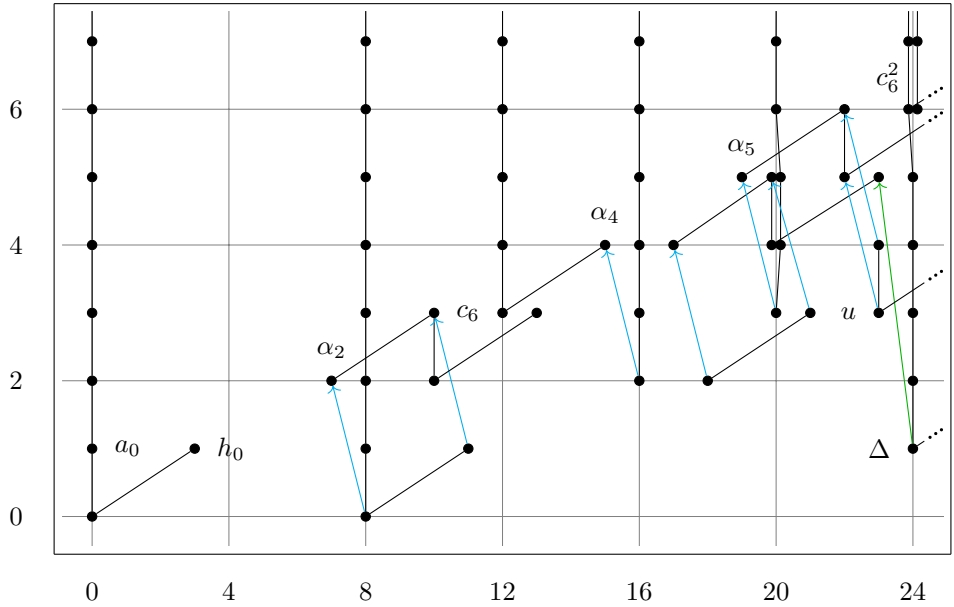

  \centering
  The $\F_3$-Adams spectral sequence for $\tmf$ \\\vspace{6pt}
  \printpage[ name = tmfASS, page = 0 ]
  \caption{The $\F_3$-Adams spectral sequence for $\tmf$ as it appears in \cite{Culvertmf}.}
  \label{fig:tmf}
\end{figure}

\begin{prop}
  The kernel of the unit map $\pi_{23}\Ss_{3} \to \pi_{23}\MOtw_{3}$ does not contain $\alpha_1 \beta_1^2$, hence is generated by the image of $J$. Equivalently, $d_5(x_{12}^2) = 0$.
\end{prop}

\begin{proof}
  We will proceed by contradiction.
  Suppose that $d_5(x_{12}^2) \dot{=} h_0 b_0^2$.
  \begin{itemize}
  \item Let $y$ denote an element of $\pi_{24}\MOtw$ which is detected by $a_0^2x_{24}$.
  \item Let $z$ denote an element of $\pi_{24}\MOtw$ which is detected by $a_0x_{12}^2$.
  \end{itemize}
  Note that any choice of $y,z$ form a basis for $\pi_{24}\MOtw$ over $\Z_3$.
  By Moss's theorem \cite{Moss}, we can choose $z$ such that
  $z \in \langle 3, \alpha \beta^2, \iota' \rangle$, where again $\iota'$ is the unit of $\MOtw$.
  Similarly, from the $d_4$ off of $\Delta$ in the $\F_3$-Adams spectral sequence for $\tmf$ we learn that
  $[3\Delta] \in \langle 3, \alpha \beta^2, \iota \rangle$.
  Post-composing with the string orientation lets us conclude that $z$ maps to $[3\Delta]$ up to higher filtration elements. 
  Then, using the fact that $c_6^2, [3\Delta]$ are generators for $\pi_{24}\tmf$ we may conclude that the map $\pi_{24}\MOtw \to \pi_{24}\tmf$ is surjective.
  Now we may choose $z$ such that it maps to $[3\Delta]$ in $\tmf$.

  Using the $\F_3$-Adams filtration of $\MOtw$ and the fact that a lift of $c_6$ is detected by $a_0^2x_{12}$, we can conclude that $c_6^2 = u_1 27z + u_2y$ for some constants $u_1 \in \ZZ_{3}^\times$ and $u_2 \in \ZZ_{3}$. Rearranging, we can write
  $$ c_6^2 - 27 u_1 z = u_2 y. $$
  Now consider the Adams filtration of the image of each side of this equality in $\tmf$.
  The left-hand side maps to $ c_6^2 - 27u_1[3\Delta] $ which has Adams filtration 5.
  The element $y$ is detected by $a_0^2x_{24}$ in filtration 5.
  However by \Cref{lemm:mo12-homology} we know that $x_{24}$ maps to zero under the map of $\mathrm{E}_2$-pages induced by the string orientation.
  Thus, the right-hand side maps to an element of Adams filtration at least 6, a contradiction.
  %% Putting these together with the fact that $c$ maps to $\tau c_6$ in tmf we get
  %% $$ \tau^2 u c_6^2 + ? \tau w \iota = \tau \wt{3}^3 [3\Delta] $$
  %% now, by a previous calculation we know $w \iota$ is zero on $C\tau$ therefore it's divisible by $\tau$. This would imply that $\tau \wt{3}^3 [3\Delta]$ is divisible by $\tau^2$, a contradiction.
\end{proof}

\section{The case of dimension $127$ at the prime $2$} \label{sec:medium}

In this section, we will prove the following proposition, which implies the dimension $127$ case of \Cref{thm:unit-ker}.

\begin{prop} \label{lem:nofilt49}
  There is no element of $\pi_{127} \Ss_{(2)}$ which is detected in $\HFt$-Adams filtration $49$. In other words, if $x \in \pi_{127} \Ss_{(2)}$ has $\F_2$-Adams filtration at least $49$, then it also has $\F_2$-Adams filtration at least $50$.
\end{prop}

\begin{proof}[Proof of \Cref{thm:unit-ker} in dimension $127$, given \Cref{lem:nofilt49}]
  By \Cref{prop:w-filt} and \Cref{tbl:remaining}, the element $w \in \pi_{127} \Ss_{(2)}$ lies in $\F_2$-Adams filtration at least $49$. By \Cref{lem:nofilt49}, $w$ in fact lies in $\F_2$-Adams filtration at least $50$. Consulting \Cref{tbl:remaining} once more, we find that $w$ must lie in the image of $J$, as desired.
\end{proof}

The proof of \Cref{lem:nofilt49} will be based on two further lemmas.

\begin{lem} \label{lem:small-E2}
  Let $\mathrm{E}_2 ^{s,t}$ denote the $\mathrm{E}_2$-page of the $\HFt$-Adams spectral sequence. Then
  \[\mathrm{E}_2 ^{49, 176} = \F_2\{h_0 ^{48} h_7\}\]
\end{lem}

\begin{proof}
  It may be read off from \cite[Fig. 5]{TangoraExt} that $\mathrm{E}_2^{17,8 0} = \F_2 \{h_0 ^{16} h_6\}$.
  It follows from \cite[Theorem 3.4.6]{greenbook} that Adams periodicity determines an isomorphism
  $\mathrm{E}_2^{17, 80} \cong \mathrm{E}_2^{49, 176}$,
  so that the latter is a one dimensional $\F_2$-vector space.
  By \cite[Lemma 3.4.15]{greenbook}, $h_0 ^{63} h_7$ is nonzero, so that
  $h_0 ^{48} h_7 \in \mathrm{E}_2^{49,176}$
  must be nonzero and so a basis for $\mathrm{E}_2^{49,176}.$
\end{proof}

\begin{rmk}
  The conclusion of this lemma can also be read off directly from Nassau's computer calculations of Ext over the Steenrod algebra \cite{Nassaucomp}.
\end{rmk}

%\begin{prop}[{\cite[Theorem 3.4.16]{GreenBook}}]
    %A generator for the image of $J$ in $\pi_{127} \Ss_{(2)}$ is detected on the $\mathrm{E}_2$-page of the $\HFt$-Adams spectral sequence by the class $h_0 ^{56} h_7$.
%\end{prop}
%
\begin{lem} \label{lem:in-J}
  Suppose that $x \in \pi_{2^n-1} \Ss_{(2)}$ is detected on the $\mathrm{E}_2$-page of the $\HFt$-Adams spectral sequence by the class $h_0 ^{2^{n-1} - 1} h_n$. Then $x$ lies in the image of the $J$ homomorphism.
\end{lem}

\begin{proof}
  By Adams vanishing \cite[Theorem 2.1]{AdamsPer}, there can be no elements above $h_0 ^{2^{n-1}-1} h_n$ in the $\HFt$-Adams spectral sequence, i.e. $\mathrm{E}^{1+s,2^{n}+s} = 0 $ for all $s \geq 2^{n-1}$.
  It therefore suffices to establish the existence of an element in the image of $J$ detected by $h_0 ^{2^{n-1}-1}h_n$. This follows from \cite[Lemma 3.4.15 and Theorem 3.4.16]{greenbook}.
\end{proof}

\begin{proof}[Proof of \Cref{lem:nofilt49}]
  Suppose that $x \in \pi_{127} \Ss_{(2)}$ were of $\HFt$-Adams filtration $49$. Then it would have to be detected on the $\mathrm{E}_2$-page by $h_0 ^{48} h_7$ by \Cref{lem:small-E2}, so that $2^{15} x$ is detected by $h_0 ^{63} h_7$.

  By \Cref{lem:in-J}, this implies that $2^{15} x$ lies in the image of $J$. Since the image of $J$ is a summand of $\pi_{127} \Ss_{(2)}$ \cite{AdamsJIV,QuillenAdamsConj}, this implies that the image of $J$ must contain an element of order $2^{16}$, which contradicts the fact that it is a cyclic group of order $2^8$ \cite{AdamsJIV, QuillenAdamsConj}.
\end{proof}

\section{Further applications} \label{sec:app}
\subsection{Stein fillable homotopy spheres} \label{sec:stein}
In this section, we complete the enumeration of odd-dimensional homotopy sphere which admit a Stein fillable contact structure, answering a question of Eliashberg \cite[3.8]{ContactWorkshop}.
%In \cite[Conjecture 5.9]{SteinFillable}, Bowden, Crowley and Stipsicz conjectured that such a contact structure exists precisely for those homotopy spheres which bound parallizable manifolds.
Bowden, Crowley and Stipsicz have constructed Stein fillable contact structures on homotopy spheres which bound parallelizable manifolds and conjectured that these are all of them \cite[Conjecture 5.9]{SteinFillable}.
We show that their conjecture is true in dimensions other than $23$. In dimension $23$, we provide a counterexample and analyze the extent to which it fails.
This result is new in dimensions $n = 23$ and $39 \leq n \leq 247$ congruent to $7$ modulo $8$, see \cite[Theorem 5.4]{SteinFillable}, \Cref{rmk:stein-known} and \cite[Theorem 3.1]{Boundaries} for the other cases.

%In \cite[Section 3.2]{Boundaries}, the authors and Jeremy Hahn established \cite[Conjecture 5.9]{SteinFillable} for dimensions

\begin{thm} \label{thm:stein}
  Let
  %$q \leq 66$,
  $q \neq 11$ be a positive integer. A homotopy sphere $\Sigma \in \Theta_{2q+1}$ admits a Stein fillable contact structure if and only if $\Sigma \in bP_{2q+2}$, i.e. if and only if the class $[\Sigma]$ of $\Sigma$ in $\coker(J)_{2q+1}$ is zero.

  On the other hand, a homotopy sphere $\Sigma \in \Theta_{23}$ admits a Stein fillable contact structure if and only if 
  \[ [\Sigma] \in \{0, \eta^3 \kappabar\} \subset \coker(J)_{23}.\]
\end{thm}

\begin{proof}
  %The only cases left open by \cite[]{SteinFillable} are when $q \equiv 3 \mod 4$ or $q = 9$.
%
  As in \cite{SteinFillable}, let $A^{U \langle q+1 \rangle}_{2q+2}$ denote the group of almost closed $(q+1)$-connected almost complex $(2q+2)$-manifolds, modulo $(q+1)$-connected almost complex cobordisms restricting to $h$-cobordisms on the boundary.
  We then have the sequence of maps
  \[A^{U\langle q+1 \rangle} _{2q+2} \to A^{\langle q+1\rangle} _{2q+2} \xrightarrow{\partial} \Theta_{2q+1} \to \coker(J)_{2q+1},\]
  and as in \cite[Proof of Theorem 5.4]{SteinFillable}, we see that an exotic sphere $\Sigma  \in \Theta_{2q+1}$ admits a Stein fillable contact structure if and only if the class $[\Sigma] \in \coker(J)_{2q+1}$ is in the image of the composite map from $A^{U\langle q+1 \rangle} _{2q+2}$.

  By \cite[Theorem 5.4]{SteinFillable}, it suffices to deal with the case when $q \equiv 3 \mod 4$ or $q = 9$.
  In the case $q=9$, \cite[Theorem 3.4(iii)]{Schultz} states that the map $A^{\langle q+1 \rangle} _{2q+2} \to \coker(J)_{2q+1}$ vanishes.
  In the case when $q \equiv 3 \mod 4$ and $q \neq 11$, this map is zero by \Cref{thm:main}. 

  On the other hand, when $q = 11$, we note that by the argument on \cite[pg. 28]{SteinFillable}, the map $A^{U\langle q+1 \rangle} _{2q+2} \to A^{\langle q+1\rangle} _{2q+2}$ is an isomorphism, as $\pi_{11} (U) \to \pi_{11} (SO)$ is an isomorphism.
  It follows that the image of the composite $A^{U\langle q+1 \rangle} _{2q+2} \to \coker(J)_{2q+1}$ is equal to $\{0, \eta^3 \kappabar\}$ by \Cref{thm:main}, as desired.
\end{proof}

\begin{rmk} \label{rmk:stein-known}
  Case (1) of \cite[Theorem 5.4]{SteinFillable} assumes that $q \neq 9$ in the $q \equiv 1 \mod 8$ case. This is because \cite[Corollary 3.2]{Schultz} does not cover this case. However, this case is in fact covered in \cite[Theorem 3.4(iii)]{Schultz}, so this hypothesis may be removed.
\end{rmk}

In particular, we obtain a counterexample to Conjecture 5.9 from \cite{SteinFillable}:

\begin{cor}
  There exists a $23$-dimensional homotopy sphere which admits a Stein fillable contact structure but does not bound a parallelizable manifold.
\end{cor}

\subsection{Mapping class groups} \label{sec:mapping}

Our results also have application to the computation of mapping class groups of highly connected manifolds. Indeed, this was the original motivation of Galatius--Randal-Williams in making their conjecture \cite{GRAbelianQuotients}.

\begin{dfn}
  Let $W^{2n} _g = \#^{g} (S^n \times S^n)$ denote the connected sum of $g$ copies of $S^n \times S^n$.
  We further let
  \[\Gamma^n _g = \pi_{0} \mathrm{Diff}^+ (W_g ^{2n} )\]
  denote the group of isotopy classes of orientation-preserving diffeomorphisms of $W^{2n}_g$.
\end{dfn}

Building on work on Kreck \cite{Kreck}, Krannich determined the group $\Gamma^n _g$ for $n \geq 3$ odd and $g \geq 1$ in terms of two extensions \cite{KrannichMappingClass}.
In the case $n \equiv 3 \mod 4$, his answer is phrased in terms of a certain exotic $(2n+1)$-sphere $\Sigma_Q$, which is the boundary of the manifold $Q$ considered in \Cref{sec:class}.
In \Cref{sec:class}, we computed $\Sigma_Q$ for all $n$.
Therefore, our results completely resolve the identity of the mysterious $\Sigma_Q$ which appeared in Krannich's work. We refer the interested reader to Krannich's paper \cite{KrannichMappingClass} for more details.

One consequence of Krannich's results is a computation of the abelianization of $\Gamma_g ^n$, extending and reproving an earlier result of Galatius and Randal-Williams \cite{GRAbelianQuotients}.\footnote{In a future version of \cite{KrannichMappingClass} that Krannich has shared with the authors, he also determines the abelianization of $\Gamma^n _g$ for even $n \geq 4$ and $g \geq 1$.} When combined with \Cref{thm:main}, \cite[Corollary E(i)]{KrannichMappingClass} implies the following result, which demonstrates the effect that our $23$-dimensional counterexample to \cite[Conjectures A \& B]{GRAbelianQuotients} can have on the abelianization of the mapping class groups of highly connected manifolds.

\begin{thm}
  Suppose that $n \geq 9$ is odd.
  Then, if $n \neq 11$ and $g \geq 3$, there is an isomorphism
  \[\mathrm{H}_1 (\Gamma^n _g) \cong \coker(J)_{2n+1} \oplus \Z/4\Z, \]
  and if $n \neq 11$ and $g = 2$, we have
  \[\mathrm{H}_1 (\Gamma^n _g) \cong \coker(J)_{2n+1} \oplus (\Z/4\Z \oplus \Z/2\Z). \]
  On the other hand, if $n = 11$ and $g \geq 3$, there is an isomorphism
  \[\mathrm{H}_1 (\Gamma^n _g) \cong \coker(J)_{23}/\eta^3 \kappabar \oplus \Z/4\Z \cong (\Z/4\Z \oplus \Z/2\Z) \oplus \Z/4\Z, \]
  and if $n = 11$ and $g = 2$, we have
  \[\mathrm{H}_1 (\Gamma^n _g) \cong \coker(J)_{2n+1}/\eta^3 \kappabar \oplus (\Z/4\Z \oplus \Z/2\Z) \cong (\Z/4\Z \oplus \Z/2\Z) \oplus (\Z/4\Z \oplus \Z/2\Z). \]
\end{thm}

In particular, one consequence of the existence of the exceptional $23$-dimensional counterexample to \cite[Conjectures A \& B]{GRAbelianQuotients} is to make the abelianization of $\Gamma^{11} _g$ smaller than would otherwise be expected.

\bibliographystyle{alpha}
\bibliography{bibliography}

\end{document}